\documentclass[11pt,reqno]{amsart}
\newcommand{\mysection}[1]{\section{#1}
      \setcounter{equation}{0}}

\newcommand\cbrk{\text{$]$\kern-.15em$]$}} 
\newcommand\opar{\text{\raise.2ex\hbox{${\scriptstyle | }$}\kern-.34em$($} }

\usepackage{color}
\usepackage{amsmath,amsthm,amssymb,amsfonts,enumerate,color,enumerate}
\usepackage[pdftex]{graphicx}

\usepackage{verbatim}
\usepackage[spanish,USenglish]{babel} 
\usepackage{color}
\usepackage{tikz}
\allowdisplaybreaks

\DeclareMathOperator*{\esssup}{\pi-ess\,sup}

\newtheorem{theorem}{Theorem}[section]
\newtheorem{lemma}[theorem]{Lemma}

\newtheorem{corollary}[theorem]{Corollary}

\theoremstyle{definition}
\newtheorem{assumption}{Assumption}[section]

\theoremstyle{remark}
\newtheorem{remark}{Remark}[section]

\newcommand\bL{\mathbb{L}}
\newcommand\bR{\mathbb{R}}

\newcommand\cB{\mathcal{B}}
\newcommand\cC{\mathcal{C}}

\newcommand\cF{\mathcal{F}}

\newcommand\cH{\mathcal{H}}

\newcommand\cL{\mathcal{L}}

\newcommand\cP{\mathcal{P}}

\newcommand\cS{\mathcal{S}}
\newcommand\cU{\mathcal{U}}
\newcommand\cV{\mathcal{V}}
\newcommand\cZ{\mathcal{Z}}

\makeatletter
 \newcommand{\sumstar}
 {\operatornamewithlimits{\sum@\kern-.2em\raise1ex\hbox{*}}}
 \makeatother

\renewcommand{\>}{\rangle}

\begin{document}

\author[I. Gy\"ongy]{Istv\'an Gy\"ongy}
\address{School of Mathematics and Maxwell Institute, 
University of Edinburgh, Scotland, United Kingdom.}
\email{i.gyongy@ed.ac.uk}

\author[S. Wu]{Sizhou Wu}
\address{School of Mathematics,
University of Edinburgh,
King's  Buildings,
Edinburgh, EH9 3JZ, United Kingdom}
\email{Sizhou.Wu@ed.ac.uk}



\begin{abstract} We present an It\^o formula for the $L_p$-norm of 
jump processes having stochastic differentials in $L_p$-spaces. The main results extend 
well-known theorems of Krylov to the case of processes with jumps, and which can be used 
to prove existence and uniqueness theorems in $L_p$-spaces for SPDEs driven by L\'evy processes. 
\end{abstract}

\title[It\^o formula]{It\^o's formula for jump processes in $L_p$-spaces}

\maketitle

\mysection{Introduction}
It\^o formulas for semimartingales taking values in function spaces 
play important roles in the theory of stochastic partial differential equations (SPDEs).   
To get a priori estimates in the $L_2$-theory of SPDEs, driven by 
a cylindrical Wiener process 
$(w^1_t,w^2_t,...)_{t\in[0,T]}$, one usually needs a suitable formula for $|u_t|^2_H$, 
the square of $H$-valued solutions $(u_t)_{t\in[0,T]}$ to SPDEs, when $H$ is a Hilbert space.   
In the framework of $L_2$-theory 
there is a Banach space $V$, embedded continuously and densely in $H$, 
such that $u_t\in V$ for $P\otimes dt$-almost every 
$(\omega,t)\in\Omega\times[0,T]$, and from the definition of the solution it follows 
that for some processes $(v^{\ast}_t)_{t\in [0,T]}$ 
and $(g^r_t)_{t\in[0,T]}$, 
with values in $V^{\ast}$ and $H$, respectively, for $r=1,2,...$, 
\begin{equation}                                      \label{eq0}
du_t=v^{\ast}_t\,dt
+g^r_t\,dw^r_t
\quad\text{for $P\otimes dt$-a.e. $(\omega,t)\in\Omega\times[0,T]$,}
\end{equation} 
where  $V^{\ast}$ 
is the adjoint of $V$. 
(Here, and later on, the summation convention with respect 
to repeated integer valued indices is used, 
i.e., $(g^r_t,\varphi)\,dw^r_t$ means $\sum_r(g^r_t,\varphi)\,dw^r_t$.) 
A basic example for such couple of spaces  $V$ and $H$ 
is the couple of Hilbert spaces $W^1_2$ and $L_2$ of real functions 
defined on the whole Euclidean space $\bR^d$. 
In this case equation \eqref{eq0} can be rewritten 
as
\begin{equation}                                                                                        \label{eq01}
du_t=D_{\alpha}f^{\alpha}_t\,dt
+g^r_t\,dw^r_t
\end{equation}
with some $L_2$-valued processes $(f^{\alpha}_t)_{t\in [0,T]}$, 
$\alpha=0,1,2,...,d$, where 
$D_{\alpha}=\frac{\partial}{\partial x^{\alpha}}$ for $\alpha=1,2,...,d$, and 
$D_{\alpha}$ is the identity operator for $\alpha=0$. 
It was first proved in \cite{P1975} that if \eqref{eq0} holds 
and $u$, $f$ and $g$ satisfy appropriate measurability 
and integrability conditions then $u$ 
has a continuous $H$-valued modification, denoted also 
by $u$, such that  $|u_t|^2_H$ has the 
stochastic differential
\begin{equation}                                                                             \label{formula1}
d|u_t|^2_H=\big(2<u_t,v^{\ast}_t>+\|g_t\|^2_H\big)\,dt
+2(u_t,g^r_t)\,dw^r_t,  
\end{equation}
where $\|h\|^{2}_{H}=\sum_r|h^r|^2_H$, and $(\,,\,)$ and $<\,,\,>$ denote  
the inner product in $H$ and the duality product of $V$ and 
$V^{\ast}$, respectively.  
The proof of this result in \cite{P1975} was combined with 
the theory of SPDEs developed 
there. A direct proof was first given in \cite{KR1981}, see also 
\cite{PR} and \cite{Ro}, and a very nice shorter proof is 
presented in \cite{K2013} when $V$ is a Hilbert space. 
 To study SPDEs driven by 
(possibly discontinuous) semimartingales,  
processes $u$ satisfying \eqref{eq0} with $dA_t$ and $dM_t$ in place 
of $dt$ and $dw_t$ were considered, and a theorem on It\^o's formula 
was proved in \cite{GK1982},  
when $A=(A_t)_{t\in[0,T]}$ 
and $M=(M^1_t,M_t^2,...)_{t\in[0,T]}$ are (possibly discontinuous) increasing 
processes and martingales, respectively.  In this situation 
a further generalisation was given in \cite{GS2017}. 
In the special case when $V=W^1_2$, $H=L_2$ and equation \eqref{eq01} holds, It\^o's 
formula \eqref{formula1} has the form 
\begin{equation}
d|u_t|^2_{L_2}=\big(2(D^{\ast}_{\alpha}u_t, f^{\alpha}_t)
+\|g_t\|^2_{L_2}\big)\,dt+2(u_t,g^r_t)\,dw^r_t,  
\end{equation}
where $D^{\ast}_{\alpha}=-D_{\alpha}$ for $\alpha=1,2,...,d$ and $D_{\alpha}^{\ast}$ 
is the identity operator for $\alpha=0$. This formula is an 
important tool in the proof of existence and uniqueness of 
solutions in $W^m_2$ Sobolev spaces for SPDEs driven by Wiener processes.  
To have the corresponding tool for solvability in $W^m_p$ spaces,  
for $p\geq2$ 
a theorem on It\^o formula for $|u_{t}|^p_{L_p}$ is proved in \cite{K2010} when 
$(u_t)_{t\in[0,T]}$ is a $W^1_p$-valued process and  
$f^{\alpha}=(f^{\alpha}_t)_{t\in[0,T]}$ 
and $(g_t^r)_{t\in[0,T]}$ in equation \eqref{eq01} are $L_p$-valued processes.  
Our aim is to present an It\^o formula for $d|u_t|^p_{L_p}$ 
when instead of \eqref{eq01} we have 
\begin{equation}                                         \label{eq02}
du_t=D_{\alpha} f^{\alpha}_t\,dt
+g^r_t\,dw^r_t+\int_{Z}h_t(z)\,d\tilde\pi_t(dz), 
\end{equation}
where $(\tilde \pi_t(dz))_{t\in[0,T]}$ is a 
Poisson martingale measure with structure measure 
$\mu(dz)$ on a measurable space $(Z,\cZ)$, and $h=(h_t)_{t\in[0,T]}$ 
is a process with 
values in 
$$
L_p(\bR^d,L_2(Z,\mu))\cap L_p(\bR^d,L_p(Z,\mu)).  
$$
Our motivation is to present 
an It\^o formula to study solvability in $L_p$-spaces of SPDEs 
driven by Wiener processes and Poisson martingale measures. 
Our main  theorems on It\^o's formula, Theorem \ref{theorem1} 
and Theorem \ref{theorem2} 
generalise Lemma 5.1 and Theorem 2.1, respectively, from \cite{K2010}. 
We use them to prove an existence and uniqueness theorem 
for a class of stochastic integro-differential equations 
in \cite{GW2019}.   In \cite{GW2019} we need an It\^o's formula 
for $d|\<u_t\>|^p_{L_p}$, where $\<u_t\>=(\sum_{i=1}^M|u^i_t|^2)^{1/2}$ and $(u^i_t)_{t\in[0,T]}$ 
is a $W^1_p$-valued process having a stochastic differential of the type 
\eqref{eq02} for each $i=1,2,...,M$.  Therefore in Theorem \ref{theorem1} of the present paper 
we consider a system of stochastic differentials  instead of a single one.  

There are well-known theorems in the literature on It\^o's 
formula for semimartingales with values 
in separable Banach spaces, see for example \cite{BNV}, \cite{R2006} and \cite{VV}. 
In some aspects these results are more general than our main theorems, but do not cover them. 
In \cite{BNV} and \cite{VV}  
only continuous semimartingales are considered and their differential 
does not contain $D_if^i\,dt$ 
terms. The semimartingales $(u_t)_{t\in[0,T]}$ in \cite{R2006} 
contains stochastic integrals  
with respect to Poisson random measures and random 
martingale measures, but the theorems on It\^o's formula in this paper cannot be applied 
to $|u_t|_{L_p}^p$.

The structure of the paper is the following. In the next section we formulate 
the main results, Theorems \ref{theorem1} and \ref{theorem2}.  
In Section \ref{section preliminaries} we 
present a suitable It\^o's formula, Theorem \ref{theorem0} 
for the $p$-th power of the norm of 
an $\bR^M$-valued semimartingale, together with a stochastic Fubini theorem  
and a very simple tool, Lemma \ref{lemma tool}, 
which allow us to prove our main results in Section \ref{section proof}, 
by adapting the ideas and methods 
of \cite{K2010}. 

In conclusion we present some notions and notations. 
All random elements are given on a fixed complete probability space 
$(\Omega,\cF,P)$ equipped with a right-continuous filtration $(\cF_t)_{t\geq0}$ 
such that $\cF_0$ contains all $P$-zero sets of $\cF$. The $\sigma$-algebra  
of the predictable subsets of $\Omega\times[0,\infty)$ is denoted by $\cP$. 
We are given a sequence 
$w=(w^1_t,w^2_t,...)_{t\geq0}$ of $\cF_t$-adapted independent Wiener processes 
$w^r=(w^r_t)_{t\geq0}$,  such that $w_t-w_s$ is independent of $\cF_s$ for any $0\leq s\leq t$. 
We are given also a Poisson random measure $\pi(dz,dt)$ on $[0,\infty)\times Z$, 
with intensity measure $\mu(dz)dt$, where $\mu$ is a $\sigma$-finite measure 
on a measurable space $(Z,\cZ)$ with a countably generated  $\sigma$-algebra 
$\cZ$. We assume that the process $\pi_t(\Gamma):=\pi((0,t]\times \Gamma)$,  
$t\geq0$, is $\cF_t$-adapted and $\pi_t(\Gamma)-\pi_s(\Gamma)$ is independent 
of $\cF_s$ for any $0\leq s\leq t$  and $\Gamma\in\cZ$ such that $\mu(\Gamma)<\infty$. 
We use the notation $\tilde\pi(dz,dt)=\pi(dz,dt)-\mu(dz)dt$ for the {\it compensated 
Poisson random measure}, and set 
$\tilde\pi_t(\Gamma)=\tilde\pi(\Gamma,(0,t])=\pi_t(\Gamma)-t\mu(\Gamma)$ for $t\geq0$ 
and $\Gamma\in \cZ$ 
such that $\mu(\Gamma)<\infty$. For basic results concerning stochastic integrals 
with respect $\pi$ and $\tilde \pi$ we refer to \cite{A2009} and \cite{IW2011}. 
The Borel $\sigma$-algebra of a topological space $V$ is denoted by $\cB(V)$. 

The space of smooth functions $\varphi=\varphi(x)$ with compact supports on 
the $d$-dimensional Euclidean space $\bR^d$ is denoted by 
$C_0^{\infty}$. For $p, q\geq 1$ we denote by $L_p=L_p(\bR^d,\bR^M)$  
and $\cL_{q}=\cL_{q}(Z, \bR^M)$ 
the Banach spaces of $\bR^M$-valued Borel-measurable functions of $f=(f^i(x))_{i=1}^M$ 
and $\cZ$-measurable functions $h=(h^{i}(z))_{i=1}^M$ of $x\in\bR^d$ and $z\in Z$, respectively 
such that 
$$
|f|_{L_p}^p=\int_{\bR^d}|f(x)|^p\,dx
\quad
\text{and}
\quad 
|h|^{q}_{\cL_{q}}=\int_{\bR^d}|h(z)|^{q}\,\mu(dz)<\infty, 
$$
where $|v|$ means the Euclidean norm for vectors $v$ from Euclidean spaces. 
The notation 
$\cL_{p,q}$ means the space $\cL_{p}\cap\cL_{q}$ with the norm 
$$
|v|_{\cL_{p,q}}=\max(|v|_{\cL_{p}},|v|_{\cL_{q}}) \quad\text{for $v\in \cL_{p}\cap\cL_{q}$}.
$$
As usual $W^1_p$ denotes the space of functions $u\in L_p$ such that 
$D_iu\in L_p$ for every $i=1,2,...,d$, where $D_iv$ means the generalised derivative 
of $v$ in $x^i$ for locally integrable functions $v$ on $\bR^d$. The norm of $u\in W^1_p$ 
is defined by 
$$
|u|_{W^1_p}=|u|_{L_p}+\sum_{i=1}^d|D_iu|_{L_p}. 
$$
The space of sequences
$\nu=(\nu^{1},\nu^{2},...)$
of vectors $\nu^{k}\in\bR^{M}$ with finite norm
$$
|\nu|_{\ell_{2} }=\big(
\sum_{k=1}^{\infty}|\nu^k|^{2}\big)^{1/2}
$$
is denoted by $\ell_2=\ell_2(\bR^M)$, and by $l_2$ when $M=1$. 
We use the notation $L_p=L_p(\ell_2)$ for $L_p(\bR^d,\ell_2)$, 
the space of Borel-measurable 
functions $g=(g^{ir})$ on $\bR^d$ with values in $\ell_2$ such that 
$$
|g|_{L_p}^p=\int_{\bR^d}|g(x)|_{\ell_2}^p\,dx<\infty.  
$$
We denote by $L_{p}=L_p(\cL_{p,q})$ and $L_p=L_p(\cL_q)$ 
the Banach spaces of Borel-measurable functions 
$h=(h^i(x,z))$ and $\tilde{h}=(\tilde{h}^i(x,z))$ of $x\in\bR^d$ 
with values in $\cL_{p,q}$ and $\cL_q$,  
respectively,  
such that 
$$
|h|^p_{L_{p}}=\int_{\bR^d}|h(x,\cdot)|^p_{\cL_{p,q}}\,dx<\infty
\quad\text{and}\quad  
|\tilde{h}|^p_{L_{p}}=\int_{\bR^d}|\tilde{h}(x,\cdot)|^p_{\cL_q}\,dx<\infty. 
$$
For fixed $T>0$, $p\geq2$ and for a separable real Banach 
space $V$ we denote by $\bL_p=\bL_p(V)$ 
the space of predictable $V$-valued functions $f=(f_t)$ of 
$(\omega,t)\in\Omega\times[0,T]$ such that 
$$
|f|_{\bL_p}^p=E\int_0^T|f_t|^p_{V}\,dt<\infty. 
$$
In the sequel $V$ will be $L_p(\bR^d,\bR^M)$, or $L_p(\bR^d,\ell_2)$, or 
$L_p(\bR^d,\cL_{p,2})$. When $V=L_p(\bR^d,\cL_{p,2})$ then for $\bL_p(V)$  
the notation $\bL_{p,2}$ is also used.  
Recall that the summation convention with respect 
to integer valued indices is used throughout 
the paper. 
\mysection{Formulation of the results}                         \label{section formulation}

\begin{assumption}                                                          \label{assumption1}
Let $u=(u^{i}_{t})_{t\in0,T}$ be a progressively measurable $L_p$-valued 
process such that there exist $f=(f^{i}_t(x))\in \bL_p$, 
$g=(g^{ir}_{t}(x))\in\bL_p$, $h=(h^i_t(x,z))\in\bL_{p,2}$, and    
an $L_p$-valued 
$\mathcal{F}_0$-measurable random variable $\psi=(\psi^i(x))_{i=1}^M$, 
such that 
for every $\varphi\in C_0^{\infty}$
\begin{equation}                                                                \label{eq1}
(u^i_t,\varphi)
=(\psi,\varphi)
+\int_0^t(f^i_s,\varphi)\,ds
+\int_0^t(g_s^{ir},\varphi)\,dw_s^r
+\int_0^t\int_Z(h^i_s(z),\varphi)\,\tilde{\pi}(dz,ds)
\end{equation}
for $P\otimes dt$-almost every $(\omega,t)\in\Omega\times[0,T]$ 
for $i=1,2,...,M$.
\end{assumption}
In equation \eqref{eq1}, and later on, we use the notation $(v,\phi)$ 
for the Lebesgue integral over $\bR^d$ of the product $v\phi$ 
for functions  $v$ and $\phi$ on $\bR^d$ when their product 
and its integral are well-defined. 

\begin{theorem}                                                          \label{theorem1}
Let Assumption \ref{assumption1} hold with $p\geq2$. Then there is 
an $L_p$-valued adapted cadlag process
 $\bar u=(\bar u^i_t)_{t\in[0,T]}$ 
such that equation \eqref{eq1}, with $\bar u$ in place of $u$, holds 
for each $\varphi\in C_0^{\infty}$  almost surely for all $t\in[0,T]$. 
Moreover,  $u=\bar u$ for $P\otimes dt$-almost every 
$(\omega,t)\in\Omega\times[0,T]$, and almost surely 
\begin{align}
|\bar u_t|^p_{L_p}
&= |\psi|_{L_p}^p
+p\int_0^t\int_{\mathbb{R}^d}|\bar u_s|^{p-2}\bar u^i_sg^{ir}_s\,dx\, dw^r_s                    \nonumber\\
& 
+\tfrac{p}{2}\int_0^t\int_{\mathbb{R}^d}\big( 2|\bar u_s|^{p-2}\bar u^i_sf^i_s
+(p-2)|\bar u_s|^{p-4}|\bar u^i_sg^{i\cdot}_s|_{l_2}^2
+|\bar u_s|^{p-2}|g_s|_{\ell_2}^2\big)\,dx\,ds                                                                \nonumber\\
& 
+p\int_0^t\int_Z\int_{\mathbb{R}^d}|\bar u_{s-}|^{p-2}\bar u^i_{s-}h^i_s
\,dx\,\tilde{\pi}(dz,ds)                                                                                              \nonumber\\
&
+\int_0^t\int_Z\int_{\mathbb{R}^d}
(|\bar u_{s-}+h_s|^p-|\bar u_{s-}|^p-p|\bar u_{s-}|^{p-2}\bar u^i_{s-}h^i_s)
\,dx\,\pi(dz,ds)                                                                                                         \label{Ito1}
\end{align}
for all $t\in[0,T]$, where   $\bar u_{s-}$ means the left-hand limit 
in $L_p$ at $s$ of $\bar u$. 
\end{theorem}

Notice that for $M=1$ equation \eqref{Ito1} has the simpler form 
$$ 
|\bar u_t|^p_{L_p}= 
 |\psi|_{L_p}^p+p\int_0^t\int_{\mathbb{R}^d}|\bar u_s|^{p-2}\
 \bar u_sg_s^r\,dx\,dw^r_s
 $$
 $$
+\tfrac{p}{2}\int_0^t\int_{\mathbb{R}^d}\big( 2|\bar u_s|^{p-2}\bar u_sf_s
+(p-1)|\bar u_s|^{p-2}|g_s|^2_{l_2}   \big)\,dx\,ds
$$
$$
+p\int_0^t\int_Z\int_{\mathbb{R}^d}
|\bar u_{s-}|^{p-2}\bar u_{s-}h_s\,dx\,\tilde{\pi}(dz,ds)
$$
\begin{equation}                                                               \label{Ito1simple}
+\int_0^t\int_Z\int_{\mathbb{R}^d}\big( | \bar u_{s-}
+h_s|^p-| \bar u_{s-}|^p-p| \bar u_{s-}|^{p-2}\bar u_{s-}h_s  \big)\,dx\,\pi(dz,ds).    
\end{equation}

\medskip

To formulate our second main theorem we take $M=1$ and make 
the following assumption.
\begin{assumption}                                                                            \label{assumption2}
Let $u=(u_{t})_{t\in0,T}$ be a progressively measurable 
$W^1_p$-valued process such that the following conditions hold:
\newline
(i)$$
E\int_0^T|u_t|^p_{W^1_p}\,dt<\infty\,; 
$$
(ii) there exist $f^{\alpha}=(f^{\alpha}_t(x))\in \bL_p$ 
for $\alpha\in\{0,1,...,d\}$, 
$g=(g^{r}_{t}(x))\in\bL_p$, $h=(h_t(x,z))\in\bL_{p,2}$, and    
an $L_p$-valued 
$\mathcal{F}_0$-measurable random variable $\psi=(\psi(x))$, 
such that 
for every  $\varphi\in C_0^{\infty}$ we have 
\begin{equation}                                                                                                \label{eq2}
(u_t,\varphi)=(\psi,\varphi)+\int_0^t(f_s^\alpha,D^*_\alpha\varphi)\,ds
+\int_0^t(g_s^r,\varphi)\,dw_s^r+\int_0^t\int_Z(h_s(z),\varphi)\,\tilde{\pi}(dz,ds)
\end{equation}
for $P\otimes dt$-almost every $(\omega,t)\in\Omega\times[0,T]$, where 
$D_{\alpha}^{\ast}=-D_{\alpha}$ for $\alpha=1,2,..,d$,  
and $D^{\ast}_{\alpha}$ is the identity operator for $\alpha=0$. 
\end{assumption}

\begin{theorem}                                                        \label{theorem2}
Let Assumption \ref{assumption2} hold with $p\geq2$. Then there is 
an $L_p$-valued adapted cadlag process
 $\bar u=(\bar u_t)_{t\in[0,T]}$ 
such that for each $\varphi\in C_0^{\infty}$ equation \eqref{eq2} holds 
with $\bar u$ in place of $u$ almost surely for all $t\in[0,T]$. 
Moreover,  $u=\bar u$ for $P\otimes dt$-almost every 
$(\omega,t)\in\Omega\times[0,T]$, and 
almost surely 
$$ 
|\bar u_t|^p_{L_p}= 
 |\psi|_{L_p}^p+p\int_0^t\int_{\mathbb{R}^d}|u_s|^{p-2}u_sg_s^r\,dx\,dw^r_s
 $$
 $$
+\tfrac{p}{2}\int_0^t\int_{\mathbb{R}^d}
\big(2|u_s|^{p-2}u_sf^0_s-2(p-1)|u_s|^{p-2}f^i_sD_iu_s
+(p-1)|u_s|^{p-2}|g_s|^2_{l_2}   \big)\,dx\,ds
$$
$$
+\int_0^t\int_Z\int_{\mathbb{R}^d}
p|\bar u_{s-}|^{p-2}\bar u_{s-}h_s\,dx\,\tilde{\pi}(dz,ds)
$$
\begin{equation}                                                                          \label{Ito2}
+\int_0^t\int_Z\int_{\mathbb{R}^d}\big( | \bar u_{s-}
+h_s|^p-| \bar u_{s-}|^p-p| \bar u_{s-}|^{p-2}\bar u_{s-}h_s  \big)\,dx\,\pi(dz,ds)   
\end{equation}
for all $t\in[0,T]$, where $\bar u_{s-}$ denotes the 
left-hand limit in $L_p(\bR^d)$ of $\bar u$ at $s\in(0,T]$. 
Furthermore, there is a constant $N=N(d,p)$ such that 
$$     
E\sup_{t\leq T}|\bar u_t|^p_{L_p} 
\leq 2 E|\psi|_{L_p}^p+NT^{p-1}E\int_0^T|f^0_t|^p_{L_p}\,dt+
NE\int_0^T|h_t|^p_{L_p(\cL_p)}\,dt
$$
\begin{equation}
+ NT^{(p-2)/2}E\int_0^T|g|^p_{L_p}+\sum_{i=1}^d|f^i_t|^p_{L_p}
+|Du_t|^p_{L_p}\,dt
+NT^{(p-2)/2}E\int_0^T|h_t|^p_{L_p(\cL_2)}\,dt.           \label{Ito Lp estimate}
\end{equation}
\end{theorem}

\mysection{Preliminaries}                                                  \label{section preliminaries}

First we present an It\^o formula for an $\bR^M$-valued semimartigale 
$X=(X^1_t,...,X^M_t)_{t\in[0,T]}$ given by 
$$
X_t=X_0+\int_0^tf_s\,ds
+\int_0^tg_s^{r}\,dw_s^r
$$
\begin{equation}                                             \label{1.9.2}
+\int_0^t\int_{Z}\bar h_s(z)\,\pi(dz,ds)
+\int_0^t\int_Zh_s(z)\,\tilde{\pi}(dz,ds), \quad\text{for $t\in[0,T]$},  
\end{equation}
where $X_0$ is 
an $\bR^M$-valued $\mathcal{F}_0$-measurable random variable, 
$f=(f^i_t)_{t\in[0,T]}$ and $g=(g^{ir}_t)_{t\in[0,T]}$ are predictable 
processes with values in $\bR^M$ and $\ell_2=\ell_2(\bR^M)$, respectively, 
$\bar h=(\bar h_t^i(z))_{t\in[0,T]}$ and $h=(h_t^i(z))_{t\in[0,T]}$ are 
$\bR^M$-valued $\cP\otimes\cZ$-measurable functions  
on $\Omega\times[0,T]\times Z$ such that almost surely
\begin{equation}                                                             \label{integrands}
\bar h_t^i(z)h_t^j(z)=0 \quad\text{for $i,j=1,2,...,M$,  
for all $t\in[0,T]$ and $z\in Z$},   
\end{equation}
and 
\begin{equation}                                                             \label{integrals}
\int_0^T\int_Z|\bar h_s(z)|\,\pi(dz,dt)<\infty,\quad 
\int_0^T|f_t|+|g_t|^2_{\ell_2}+|h_t(\cdot)|^{2}_{\cL_2}\,dt<\infty. 
\end{equation}
\begin{theorem}                                                             \label{theorem0}
Let conditions \eqref{integrands} and \eqref{integrals} hold, 
and let $\phi$ from $C^2(\bR^M)$, the space of continuous 
real functions on $\bR^M$ whose derivatives up 
to second order are continuous functions on $\bR^M$. 
Then $\phi(X_t)$ is a semimartingale such that 
\begin{align}
\phi(X_t)
= &\phi(X_0)+ \int_0^tD_i\phi(X_s)g_s^{ir}\,dw_s^r                        
+\int_0^tD_i\phi(X_s)f^i_s+
\tfrac{1}{2}D_iD_j\phi(X_s)g_s^{ir}g_s^{jr}\,ds                                                                    \nonumber\\
& + \int_0^t\int_Z\phi(X_{s-}+\bar h_s(z))-\phi(X_{s-})\,\pi(dz,ds)
+\int_0^t\int_ZD_i\phi(X_{s-})h^i_s(z)\,\tilde\pi(dz,ds)                                             \nonumber\\
&+\int_0^t\int_Z
\phi(X_{s-}+h_s(z))-\phi(X_{s-})-D_i\phi(X_{s-})h^i_s(z)\,\pi(dz,ds)                \label{Ito01}
\end{align}
almost surely for all $t\in[0,T]$. 
\end{theorem}
In this paper we need the following corollary of this theorem. 
\begin{corollary}                                                        \label{corollary Ito}
Let conditions \eqref{integrands} and \eqref{integrals} hold. Then for 
any $p\geq2$ the process $|X_t|^p$ is a semimartingale such that  
\begin{align}
|X_t|^p
= &|\psi|^p+ p\int_0^t|X_s|^{p-2}X^i_sg_s^{ir}\,dw_s^r                           \nonumber\\
& + \tfrac{p}{2}\int_0^t\left(2|X_s|^{p-2}X^i_sf^i_s+
(p-2)|X_s|^{p-4}|X^i_sg^{i\cdot}_s|_{l_2}^2
+|X_s|^{p-2}|g_s|_{l_2}^2\right)\,ds                                                                    \nonumber\\
& + p\int_0^t\int_Z|X_{s-}|^{p-2}X^i_{s-}h^i_s(z)\,\tilde\pi(dz,ds)
+\int_0^t\int_Z(|X_{s-}+\bar{h}_s|^p-|X_{s-}|^p)\,\pi(dz,ds)                \nonumber\\
&+\int_0^t\int_Z\left(
|X_{s-}+h_s|^p-|X_{s-}|^p-p|X_{s-}|^{p-2}X^i_{s-}h^i_s\right)\,\pi(dz,ds)                \label{Itop}
\end{align}
almost surely for all $t\in[0,T]$. 
\end{corollary}
\begin{proof}
Since the function $\phi(x)=|x|^p$ for $p\geq2$ belongs to $C^2(\bR^M)$  
with 
$$
D_i|x|^p=p|x|^{p-2}x^i, \quad D_jD_i|x|^p=p(p-2)|x|^{p-4}x^ix^j+p|x|^{p-2}\delta_{ij},  
$$
it is easy to see that Theorem \ref{theorem0} for $\phi(x)=|x|^p$ gives the corollary. 
Here and in the sequel $0/0:=0$. 
\end{proof}
We obtain Theorem \ref{theorem0} from 
the following well-known theorem on It\^o's formula. 

\begin{theorem}                                                                               \label{theorem standard}
Besides conditions \eqref{integrands} and 
\eqref{integrals} assume there is a constant $K$ such that $|h|\leq K$ for all 
$\omega\in\Omega$, $t\in[0,T]$ and $z\in Z$. Then for any 
$\phi\in C^2(\bR^M)$ the process $(\phi(X_t))_{t\in[0,T]}$ is a semimartingale 
such that 
$$
\phi(X_t)=\phi(X_0)
+\int_0^tf^i_sD_i\phi(X_s)+\tfrac{1}{2}g_s^{ir}g_s^{jr}D_{i}D_{j}\phi(X_s)\,ds
+\int_0^tg^{ir}_sD_i\phi(X_s)\,dw^r_s
$$
$$
+\int_0^t\int_Z\phi(X_{s-}+\bar h_s(z))-\phi(X_{s-})\,\pi(dz,ds)
+\int_0^t\int_Z\phi(X_{s-}+h_s(z))-\phi(X_{s-})\,\tilde \pi(dz,ds)
$$
\begin{equation}                                                                       \label{formula standard}
+\int_0^t\int_Z
\left(
\phi(X_s+h_s(z))-\phi(X_s)-h^i_s(z)D_i\phi(X_s)
\right)
\,\mu(dz)\,ds. 
\end{equation}
\end{theorem} 
\begin{proof}
This theorem, with a finite dimensional Wiener process 
in place of an infinite sequence of independent Wiener processes   
is proved, for example, in \cite{IW2011}. The extension 
of it to our setting is a simple exercise left for the reader. 
\end{proof}
Notice that for $\phi(x)=|x|^p$ 
the last two integrals in \eqref{formula standard} may not exist  
without the additional condition that  
$h$ is bounded. Thus It\^o's formula \eqref{formula standard} 
does not hold in general for $\phi(x)=|x|^p$, $p\geq2$, 
under the conditions \eqref{integrands} and \eqref{integrals}. 

We prove Theorem \ref{theorem0} by rewriting 
It\^o formula \eqref{formula standard} 
into equation \eqref{Ito01} 
under the additional condition that $h$ is bounded, and  
we dispense with this condition by approximating 
$h$ by bounded functions.

\begin{proof}[Proof of Theorem \ref{theorem0}]
First in addition to the conditions \eqref{integrands} and 
\eqref{integrals} assume there is a constant $K$ such that 
$|h|\leq K$. By Taylor's formula 
for 
$$
I^{a}\phi(v):=\phi(v+a)-\phi(v)
\quad\text{and}\quad 
J^{a}\phi(v):
=I^{a}\phi(v)-D_i\phi(v)a^i,
$$ 
for each $v,a\in\bR^M$ we have 
\begin{equation}                                                               \label{Taylor}
|I^a\phi(v)|\leq \sup_{|x|\leq |a|+|v|}|D\phi(x)||a|, 
\quad |J^a\phi(v)|\leq \sup_{|x|\leq |a|+|v|}|D^2\phi(x)||a|^2,  
\end{equation}
where $|D\phi|^2:=\sum_{i=1}^M|D_i\phi|^2$ and 
$|D^2\phi|^2:= \sum_{i=1}^M\sum_{j=1}^M|D_iD_j\phi|^2$. 
Since $(X_t)_{t\in[0,T]}$ is a cadlag process, $R:=\sup_{t\leq T}|X_t|$ 
is a finite random variable. Thus we have  
\begin{equation}                                                             \label{Jintegral1}
\int_0^T\int_Z|J^{h_t(z)}\phi(X_{t-})|\mu(dz)\,dt
\leq \sup_{|x|\leq R+K}|D^2\phi(x)
|\int_0^T\int_Z|h_t(z)|^2\,\mu(dz)\,dt<\infty 
\end{equation}
and 
\begin{equation}                                                            \label{Jintegral2}
\int_0^T\int_Z|J^{h_t(z)}\phi(X_{t-})|^2\mu(dz)\,dt
\leq \sup_{|x|\leq R+K}|D^2\phi(x)|^2K^2
\int_0^T\int_Z|h_t(z)|^2\,\mu(dz)\,dt<\infty 
\end{equation}
almost surely.  Clearly, 
$$
\int_0^T\int_Z|D_i\phi(X_{t-})h^i_t(z)|^2\,\mu(dz)\,dt
\leq \sup_{|x|\leq R}|D\phi(x)|^2
\int_0^T\int_Z|h_t(z)|^2\,\mu(dz)\,dt<\infty \,\,(\rm{a.s.}).   
$$
Hence, by virtue of \eqref{Jintegral2} the stochastic It\^o integral 
$$
\int_0^t\int_Z\phi(X_{t-}+h_t(z))-\phi(X_t)\,\tilde\pi(dz,dt)
=\int_0^t\int_ZI^{h_t(z)}\phi(X_{t-})\tilde\pi(dz,dt)
$$
can be decomposed as
$$
\int_0^t\int_ZI^{h_t(z)}\phi(X_{t-})\tilde\pi(dz,dt)
=\int_0^t\int_ZJ^{h_t(z)}\phi(X_{t-})\,\tilde\pi(dz,dt)
+\int_0^t\int_ZD_i\phi(X_{t-})h^i_t(z)\,\tilde\pi(dz,dt), 
$$ 
and by virtue of \eqref{Jintegral1} and \eqref{Jintegral2}, 
$$
\int_0^t\int_ZJ^{h_t(z)}\phi(X_{t-})\,\tilde\pi(dz,dt)
+\int_0^t\int_ZJ^{h_t(z)}\phi(X_{t-})\, \mu(dz)\,dt
=\int_0^t\int_ZJ^{h_t(z)}\phi(X_{t-})\,\pi(dz,dt). 
$$
Hence
$$
\int_0^t\int_ZI^{h_t(z)}\phi(X_{t-})\,\tilde\pi(dz,dt)+
\int_0^t\int_ZJ^{h_t(z)}\phi(X_{t-})\, \mu(dz)\,dt
$$
$$
=\int_0^t\int_ZD_i\phi(X_{t-})h^i_t(z)\,\tilde\pi(dz,dt)
+\int_0^t\int_ZJ^{h_t(z)}\phi(X_{t-})\,\pi(dz,dt), 
$$
which shows that Theorem \ref{theorem0} holds 
under the additional condition that $|h|$ is bounded. 
To prove the theorem in full generality we approximate 
$h$ by $h^{(n)}=(h^{1n},...,h^{Mn})$,  where 
$h_t^{in}=-n\vee h_t^{i}\wedge n$ for integers $n\geq1$, and define 
$$
X^{(n)}_t:=X_0+\int_0^tf_s\,ds+\int_0^tg_s^{r}\,dw_s^r
+\int_0^t\int_Z\bar h_s(z)\,\pi(dz,ds)
+\int_0^t\int_Zh^{(n)}_s(z)\,\tilde{\pi}(dz,ds), \quad t\in[0,T]. 
\quad
$$
Clearly, for all $(\omega,t,z)$ 
\begin{equation}                                                      \label{h}
|h^{(n)}|\leq \min(|h|, nM)\quad
\text{and \quad $h^{(n)}\rightarrow h$\quad 
as $n\rightarrow \infty$}. 
\end{equation}
Therefore Theorem \ref{theorem0} for $X^{(n)}$ holds, and 
$$
\lim_{n\to\infty}\int_0^T\int_Z|h^{(n)}_t(z)-h_t(z)|^2\,\mu(dz)\,dt=0\,\,(\rm{a.s.}), 
$$
which implies 
$$
\sup_{t\leq T}|X^{(n)}_t-X_t|\to 0 \quad\text{in probability as $n\to\infty$.}
$$
Thus there is a strictly increasing subsequence 
of positive integers $(n_k)_{k=1}^{\infty}$  such that 
$$
\lim_{k\to\infty}\sup_{t\leq T}|X^{(n_k)}_t-X_t|=0\quad(\rm{a.s.}), 
$$ 
which implies  
$$
\rho:=\sup_{k\geq1}\sup_{t\leq T}|X^{(n_k)}_t|<\infty\quad (\rm{a.s.}). 
$$
Hence it is easy to pass to the limit $k\to\infty$ in $\phi(X_t^{(n_k)})$ 
and in the first two integral terms in the equation for $\phi(X_t^{(n_k)})$ in 
Theorem \ref{theorem0}. To pass to the limit in the other terms in this equation  
notice that since $\pi(dz,dt)$ is a counting measure of a point process, 
from the condition 
for $\bar h$  in \eqref{integrals} we get 
\begin{equation}                                                                    \label{esssup1}
\xi:={\esssup}\,|\bar h|<\infty\,\,(\rm{a.s.}),  
\end{equation}
where ${\esssup}$ denotes the essential supremum operator 
with respect to the measure $\pi(dz,dt)$ over $Z\times[0,T]$. 
Similarly, from the condition for $h$ we have
\begin{equation}                                                                     \label{esssup2}
\eta:={\esssup}\,|h|<\infty\,\,(\rm{a.s.}).   
\end{equation}
This can be seen by noting that 
for the sequence 
of predictable stopping times 
$$
\tau_j=\inf\left\{t\in[0,T]:\int_0^t\int_Z|h_s(z)|^2\,\mu(dz)\,ds\geq j \right\}, 
\quad\text{$j=1,2,...$}, 
$$
we have 
$$
E\int_0^T\int_Z{\bf1}_{t\leq\tau_j}|h_t(z)|^2\,\pi(dz,dt)
=E\int_0^T\int_Z{\bf1}_{t\leq\tau_j}|h_t(z)|^2\,\mu(dz)\,dt\leq j<\infty, 
$$
which gives 
$$
\int_0^T\int_Z|h_t(z)|^2\,\pi(dz,dt)<\infty 
\quad
\text{almost surely on 
$\Omega_j=\{\omega\in\Omega:\tau_j\geq T\}$ for each $j\geq1$.} 
$$
Since $(\tau_j)_{j=1}^{\infty}$ is an increasing sequence 
converging to infinity, we have $P(\cup_{j=1}^{\infty}\Omega_j)=1$, i.e., 
\begin{equation}                                                                      \label{pi}
\int_0^T\int_Zh^2_t(z)\,\pi(dz,dt)<\infty\,\,(\rm{a.s.})   
\end{equation}
which implies \eqref{esssup2}. 
 By \eqref{esssup1} and 
the first inequality 
in \eqref{Taylor}, we have  
$$
|I^{\bar h_t(z)}\phi(X_{t-}^{(n_k)})|+|I^{\bar h_t(z)}\phi(X_{t-})|
\leq 2\sup_{|x|\leq \rho+\xi}|D\phi(x)||\bar h_t(z)|<\infty
$$
almost surely for $\pi(dz,dt)$-almost every $(z,t)\in Z\times[0,T]$. 
Hence by Lebesgue's theorem on dominated convergence we get 
$$
\lim_{k\to\infty}\int_0^T\int_Z
|I^{\bar h_s(z)}\phi(X^{(n_k)}_{s-})-I^{\bar h_s(z)}\phi(X_{s-})|\,\pi(dz,ds)=0 \quad \rm{(a.s.)},  
$$
which implies that for $k\to\infty$ 
$$
\int_0^t\int_Z
I^{\bar h_s(z)}\phi(X^{(n_k)}_{s-})\,\pi(dz,ds)
\to
\int_0^t\int_Z I^{\bar h_s(z)}\phi(X_{s-})\,\pi(dz,ds)
$$
almost surely, uniformly in $t\in[0,T]$. Clearly, 
$$
|D_i\phi(X^{(n_k)}_{t-})h^{in_k}_t(z)|^2+|D_i\phi(X_{t-})h^{i}_t(z)|^2
\leq 2\sup_{|x|\leq \rho}|D\phi(x)|^2|h_t(z)|^2
$$
almost surely for all $(z,t)\in Z\times[0,T]$. Hence by Lebesgue's theorem 
on dominated convergence 
$$
\lim_{k\to\infty}\int_0^T\int_Z
|D_i\phi(X^{(n_k)}_{t-})h^{in_k}_t(z)-D_i\phi(X_{t-})h^{i}_t(z)|^2
\,\mu(dz)\,dt=0\quad \rm{(a.s.)},  
$$
which implies that for $k\to\infty$ 
$$
\int_0^t\int_Z
D_i\phi(X^{(n_k)}_{t-})h^{in_k}_t(z)\,\tilde\pi(dz,dt)
\to
\int_0^t\int_Z
D_i\phi(X_{t-})h^{i}_t(z)\,\tilde\pi(dz,dt)   
$$
in probability, uniformly in $t\in[0,T]$. 
Finally note that by using the second inequality in \eqref{Taylor} 
together with \eqref{esssup2} we have  
$$
|J^{h^{(n_k)}_t(z)}\phi(X_{t-}^{(n_k)})|+|J^{h_t(z)}\phi(X_{t-})|
\leq 2\sup_{|x|\leq \rho+\eta}|D^2\phi(x)||h_t(z)|^2
$$
almost surely for $\pi(dz,dt)$-almost every $(z,t)\in Z\times[0,T]$. Hence, 
taking into account \eqref{pi}, by Lebesgue's theorem on dominated convergence 
we obtain  
$$
\lim_{k\to\infty}\int_0^T\int_Z
|J^{h^{(n_k)}_t(z)}\phi(X^{(n_k)}_{t-})-J^{h_t(z)}\phi(X_{t-})|
\,\pi(dz,dt)=0\quad \rm{(a.s.)},  
$$
which implies that for $k\to\infty$ 
$$
\int_0^t\int_ZJ^{h^{(n_k)}_t(z)}\phi(X^{(n_k)}_{t-})\,\pi(dz,dt)
\to
\int_0^t\int_ZJ^{h_t(z)}\phi(X_{t-})\,\pi(dz,dt)
$$
almost surely, uniformly in $t\in[0,T]$, and finishes the proof of the theorem.  
\end{proof}
\begin{remark}
One can give a different proof of Theorem \eqref{theorem0} 
by showing that for finite measures $\mu$, 
the It\^o's formula 
for general semimartingales, Theorem VIII.27 in \cite{DM1982}, 
 applied to $(X_t)_{t\in[0,T]}$, can be rewritten as 
equation \eqref{Ito01}. Hence by an approximation procedure 
one can get the general case of $\sigma$-finite measures $\mu$. 
\end{remark}
To obtain Theorem \ref{theorem1} from Theorem \ref{theorem0} besides 
well-known Fubini theorems for deterministic integrals and stochastic integrals 
with respect to Wiener processes, see \cite{K2011}, 
we need the following Fubini theorems for stochastic integrals 
with respect to Poisson random measures and 
Poisson martingale measures, where $(\Lambda,\mathcal{S},m)$ 
denotes a measure space, with a $\sigma$-finite measure $m$ 
and a countably generated $\sigma$-algebra $\cS$.

\begin{theorem}                                                                    \label{theorem Fubini1}
Let $f=f(\omega,t,z,\lambda)$ 
be a 
$\mathcal{P}\otimes \mathcal{Z}\otimes \mathcal{S}$-measurable 
real function on $\Omega\times[0,T]\times Z\times\Lambda$ 
such that
\begin{equation}                                                                       \label{condition1}
\int_0^T\int_Z|f(t,z,\lambda)|^2\,\mu(dz)\,dt <\infty
\end{equation}
for every $\lambda\in\Lambda$ and $\omega\in\Omega$. 
Then there is an $\cF\otimes\cB([0,T])\otimes\cS$-measurable function 
$F=F(t,\lambda)$ such that it is cadlag in $t\in[0,T]$ for 
every $(\omega,\lambda)\in\Omega\times\Lambda$, 
for each $\lambda\in\Lambda$ the process $(F(t,\lambda))_{t\in[0,T]}$ 
is a locally square-integrable $\cF_t$-martingale  
and
\begin{equation}                                                              \label{regular}
F(t,\lambda)=\int_0^t\int_Zf(s,z,\lambda)\,\tilde{\pi}(dz,ds)
\quad\text{almost surely for all $t\in[0,T]$}.
\end{equation}
Moreover, if almost surely 
\begin{equation}                                                           \label{Fubini cond}
\int_\Lambda
\left( \int_0^T\int_Z|f(t,z,\lambda)|^2\,\mu(dz)\,dt 
\right)^{1/2}\,m(d\lambda)<\infty, 
\end{equation}
then almost surely
\begin{equation}                                                             \label{Fubini eq}
\int_\Lambda F(t,\lambda)\,m(d\lambda)
=\int_0^t\int_Z\int_\Lambda 
f(s,z,\lambda)\,m(d\lambda)\,\tilde{\pi}(dz,ds)
\quad\text{ for all $t\in[0,T]$}.
\end{equation}
\end{theorem}
\begin{proof} The proof of this theorem is similar to that of Lemma 2.5 from \cite{K2011}. 
Let us call the function $F$, whose existence is stated in the theorem, a regular version of 
the stochastic integral process defined in the right-hand side of \eqref{regular}. 
Assume first that $\mu$ and $m$ are finite measures, and consider the 
space $\cH$ of $\mathcal{P}\otimes \mathcal{Z}\otimes \mathcal{S}$-measurable 
bounded real functions $f$ such that the conclusions of the theorem hold. 
Then it is easy to see that $\cH$ is a real vector space which contains the constants. 
Let $(f^n)_{n=1}^{\infty}$ be an increasing uniformly bounded sequence from $\cH$,   
and denote by $F^n$ the regular version of the stochastic integral of $f^n$. 
Thus, in particular, for each $\lambda\in\Lambda$ 
\begin{equation}                                                       \label{regularn}
F^n(t,\lambda)=\int_0^t\int_Zf^n(s,z,\lambda)\,\tilde{\pi}(dz,ds)
\quad\text{almost surely for all $t\in[0,T]$,}
\end{equation}
and almost surely 
\begin{equation}                                                             \label{nequation}
\int_\Lambda F^n(t,\lambda)\,m(d\lambda)
=\int_0^t\int_Z\int_\Lambda 
f^n(s,z,\lambda)\,m(d\lambda)\,\tilde{\pi}(dz,ds)
\quad\text{ for all $t\in[0,T]$}. 
\end{equation}
Set $f=\lim_{n\to\infty}f^n$. Then $f$ is a bounded 
$\mathcal{P}\otimes \mathcal{Z}\otimes \mathcal{S}$-measurable 
function and 
$$
\lim_{n\to\infty}\int_0^T\int_Z(f^n(t,z,\lambda)-f(t,z,\lambda))^2\,\mu(dz)\,dt=0
\quad
\text{for every $\lambda\in\Lambda$.}
$$
Consequently, for each $\lambda\in\Lambda$ 
the sequence $F^n(t,\lambda)$ converges 
in probability, uniformly in $t\in[0,T]$, and hence 
by a straightforward modification 
of Lemma 2.1 from \cite{K2011} there is a 
$\cF\otimes\cB([0,T])\otimes\cS$-measurable function 
$F=F(t,\lambda)$ such that it is cadlag in $t\in[0,T]$ for 
every $(\omega,\lambda)\in\Omega\times\Lambda$, 
for each $\lambda\in\Lambda$ the process $(F(t,\lambda))_{t\in[0,T]}$ 
is a locally square-integrable $\cF_t$-martingale,   
and \eqref{regular} holds. Now we show that almost surely \eqref{Fubini eq} 
also holds, by taking $n\to\infty$ in equation \eqref{nequation}. 
Clearly, by Lebesgue's theorem on dominated convergence we have 
$$
\lim_{n\to\infty}
\int_0^T\int_Z\left(\int_{\Lambda}(f^n-f)(t,z,\lambda)\,m(d\lambda)\right)^2
\,\mu(dz)\,dt=0
$$
for every $\omega\in\Omega$, which implies that for 
$n\to\infty$ 
\begin{equation}                            \label{conv1}
\int_0^t\int_Z\int_\Lambda 
f^n(s,z,\lambda)\,m(d\lambda)\,\tilde{\pi}(dz,ds)\to 
\int_0^t\int_Z\int_\Lambda 
f(s,z,\lambda)\,m(d\lambda)\,\tilde{\pi}(dz,ds)
\end{equation}
in probability, uniformly in $t\in[0,T]$. By the Davis inequality 
$$
E\int_{\Lambda}\sup_{t\in[0,T]}|F^n-F|(t,\lambda)\,m(d\lambda)
\leq 3\int_{\Lambda}E\left(\int_0^T\int_Z|f^n(t,z)-f(t,z)|^2\,\mu(dz)\,dt\right)^{1/2}
\,m(d\lambda), 
$$
and the right-hand side of this inequality converges to zero by virtue of 
Lebesgue's theorem on dominated convergence again. Hence 
for $n\to\infty$ 
\begin{equation}                                \label{conv2}
\int_{\Lambda}F^n(t,\lambda)\,m(d\lambda)\to 
\int_{\Lambda}F(t,\lambda)\,m(d\lambda)
\quad
\text{in probability, uniformly in $t\in[0,T]$},  
\end{equation}
and equation \eqref{Fubini eq} follows. 
Thus we have proved that if $f$ is the limit of an 
increasing uniformly bounded sequence of functions $f^n$ from $\cH$ 
then $f$ belongs to $\cH$. Let $\cC$ denote the class of functions $f$ 
of the form $f(t,z,\lambda)=c{\bf1}_{(r,s]}{\bf1}_U\varphi(\lambda)$, 
for  $0\leq r\leq s\leq T$, bounded $\cF_r$-measurable random 
variables $c$, sets $U\in\cZ$ and bounded $\cS$-measurable 
real functions $\varphi$. Then 
$$
\int_0^t\int_Zf(s,z,\lambda)\,\tilde\pi(dz,ds)=c\varphi(\lambda)
\tilde\pi((r\wedge t, s\wedge t]\times U)=:F(t,\lambda), \quad t\in[0,T],\,\lambda\in\Lambda 
$$
is a regular version of the stochastic integral of $f$, 
and it is easy to see that \eqref{Fubini eq} 
holds. Notice that $\cC$ is closed with respect to the multiplication of functions, 
and the $\sigma$-algebra generated by  
$\cC$ on $\Omega\times[0,T]\times Z\times\Lambda$
 is $\cP\otimes\cZ\otimes\cS$. 
Consequently, by the well-known Monotone Class Theorem, $\cH$ contains 
all $\mathcal{P}\otimes \mathcal{Z}\otimes \mathcal{S}$-measurable bounded 
real functions on $\Omega\times[0,T]\times Z\times\Lambda$. 

Consider now a  
$\mathcal{P}\otimes \mathcal{Z}\otimes \mathcal{S}$-measurable function $f$ 
satisfying \eqref{condition1}, and for every integer $n\geq1$  
define $f^n=-n\vee f\wedge n$. Then clearly, $f^n$ is bounded, 
$\mathcal{P}\otimes \mathcal{Z}\otimes \mathcal{S}$-measurable, and satisfies 
\eqref{condition1} and \eqref{Fubini cond}. Hence by virtue of what we have proved 
above, there is a regular version $F^n$ of the stochastic integral process of $f^n$, 
i.e., in particular, with this $F^n$ and $f^n=-n\vee f\wedge n$ equations 
\eqref{regularn} and \eqref{nequation} hold. Clearly, 
$\lim_{n\to\infty}f^n=f$ and $|f^n-f|\leq |f|$ for 
all $\omega\in\Omega$, $t\in[0,T]$, $z\in Z$ and $\lambda\in\Lambda$,  
which allow us to repeat the above arguments to show the existence 
of a regular version $F$ for the stochastic integral process of $f$,  
and to get  \eqref{conv1} if \eqref{Fubini cond} also holds.  
To obtain also  \eqref{conv2}, under the additional assumption 
\eqref{Fubini cond}, we introduce 
the process
$$
Q(t)=\int_{\Lambda}
\left(\int_0^t\int_Zf^2(s,z,\lambda)\,\mu(dz)\,ds\right)^{1/2}\,m(d\lambda), 
\quad t\in[0,T], 
$$
and the random time $\tau_{\delta}=\inf\{t\in[0,T]: Q(t)\geq\delta\}$ for $\delta>0$. 
Then $Q$ is a continuous $\cF_t$-adapted process.  Thus $\tau_{\delta}$ 
is an $\cF_t$-stopping time and for every $\omega\in\Omega$ 
$$
Q_n(t):=\int_{\Lambda}\left(\int_0^t\int_Z|f-f^n|^2(s,z,\lambda)\,
\mu(dz)\,ds\right)^{1/2}\,m(d\lambda)
$$
\begin{equation}                                          \label{Q}
\leq Q(t)
\leq\delta\quad\text{for $t\leq T\wedge\tau_{\delta}$}. 
\end{equation}
Hence for any $\varepsilon>0$ by the Markov and Davis inequalities 
$$
P\left(\int_{\Lambda}\sup_{t\leq T}|F^n-F|(t,\lambda)\,m(d\lambda)\geq \varepsilon\right)
$$
$$
\leq P\left(
\int_{\Lambda}\sup_{t\leq T}|F^n-F|(t\wedge\tau_{\delta},\lambda)\,m(d\lambda)
\geq \varepsilon 
\right)
+P\left(\tau_{\delta}<T\right)
$$
\begin{equation}                                                                \label{Qn}
\leq \varepsilon^{-1}E(Q_n(T)\wedge\delta)
+P\left(Q(T)\geq\delta\right). 
\end{equation}
Letting here first $n\to\infty$ and then $\delta\to\infty$ we get 
\begin{equation}                                                                    \label{convn}
\lim_{n\to\infty}
P\left(\int_{\Lambda}\sup_{t\leq T}
|F^n-F|(t,\lambda)\,m(d\lambda)\geq \varepsilon\right)
=0\quad\text{for any $\varepsilon>0$}, 
\end{equation}
which implies \eqref{conv2} in the new situation, 
and finishes the proof of the theorem 
under the additional condition that $\mu$ and $m$ are finite measures.  

In the general case 
of $\sigma$-finite measures $\mu$ and $m$ for every integer $n\geq1$ we define    
$\tilde \pi_n$, $\mu_n$ and $m_n$ by
$$
\tilde \pi_n(F)=\tilde \pi(F\cap(Z_n\times(0,T])), \quad \mu_n(A)=\mu(A\cap Z_n), 
\quad m_n(B)=m(B\cap\Lambda_n)
$$
for $F\in\cZ\otimes\cB(0,T)$, $A\in\cZ$ and $B\in\cS$, 
where $Z_n\in\cZ$ and $\Lambda_n\in\cS$ are sets 
such that $\mu(Z_n)<\infty$, $m(\Lambda_n)<\infty$, $Z_n\subset Z_{n+1}$, 
$\Lambda_n\subset \Lambda_{n+1}$ 
for every $n\geq1$, and  
$\cup_{n}Z_n=Z$ and $\cup_{n}\Lambda_n=\Lambda$. 
It is easy to see that $\tilde \pi_n$ 
is a Poisson martingale measure with characteristic measure $\mu_n$.  
Let $f$ be a 
$\mathcal{P}\otimes \mathcal{Z}\otimes \mathcal{S}$-measurable function 
satisfying \eqref{condition1}. Then clearly, $f$ satisfies \eqref{condition1} also 
with $\mu_n$ in place of $\mu$ and $Z_n$ in place of $Z$.  Hence 
by what we have already proved above, there is a regular 
version $F_n(t,\lambda)$ of the stochastic integral of 
$f$ (over $Z_n\times(0,t]$) 
with respect to $\tilde\pi_n$, i.e., in particular,  
for each $\lambda\in\Lambda_n$ 
$$
F_n(t,\lambda)=\int_0^t\int_{Z_n}f(s,z,\lambda)\,\tilde\pi_n(dz,ds)
=\int_0^t\int_{Z}{\bf1}_{Z_n}f(s,z,\lambda)\,\tilde\pi(dz,ds) 
$$
almost surely for all $t\in[0,T]$, 
and if $f$ satisfies also \eqref{Fubini cond}, then almost surely 
$$
\int_{\Lambda_n}F_n(t,\lambda)\,m_n(d\lambda)
=\int_0^t\int_{Z_n}\int_{\Lambda_n}f(s,z,\lambda)\,m_n(d\lambda)\,\tilde\pi_n(dz,ds)
$$
\begin{equation}                                                                 \label{nFubini}
=\int_0^t\int_{Z}\int_{\Lambda}{\bf1}_{Z_n}{\bf1}_{\Lambda_n}f(s,z,\lambda)
\,m(d\lambda)\,\tilde\pi(dz,ds)
\quad\text{for all $t\in[0,T]$.}
\end{equation}
Clearly, $\lim_{n\to\infty}f{\bf1}_{Z_n}=f$ and $|f-f{\bf1}_{Z_n}|\leq |f|$ for 
all $\omega\in\Omega$, $t\in[0,T]$, $z\in Z$ and $\lambda\in\Lambda$. 
Hence 
$$
\lim_{n\to\infty}\int_0^T\int_{Z}(f-f{\bf1}_{Z_n})^2(s,z,\lambda)\,\mu(dz)\,ds=0
\quad
\text{for all $\omega\in\Omega$, and $\lambda\in\Lambda$},  
$$
which, just like before, implies the existence of a 
regular version $F(t,\lambda)$ of the 
stochastic integral of $f$ with respect to $\tilde\pi$ (over $Z\times(0,t]$). 
If $f$ satisfies also \eqref{Fubini cond},  then we have  
$$
\lim_{n\to\infty} 
\int_0^T\int_{Z}
\left(
\int_{\Lambda}(f-{\bf1}_{Z_n}{\bf1}_{\Lambda_n}f)(s,z,\lambda)
\,m(d\lambda)
\right)^2
\,\mu(dz)\,ds=0\quad\text{for all $\omega\in\Omega$},  
$$
which implies 
$$
\int_0^t\int_{Z}\int_{\Lambda}{\bf1}_{Z_n}{\bf1}_{\Lambda_n}f(s,z,\lambda)
\,m(d\lambda)\,\tilde\pi(dz,ds)\to 
\int_0^t\int_{Z}\int_{\Lambda}f(s,z,\lambda)
\,m(d\lambda)\,\tilde\pi(dz,ds)
$$
in probability, uniformly in $t\in[0,T]$, as $n\to\infty$. 
Introducing the stopping 
time $\tau_{\delta}$ as before, we have \eqref{Q}, \eqref{Qn}, and hence 
\eqref{convn} with $F_n{\bf1}_{\Lambda_n}$ 
and $f{\bf1}_{\Lambda_n}{\bf1}_{Z_n}$ in place of $F^n$ 
and $f^n$, respectively. Consequently, letting $n\to\infty$ in 
\eqref{nFubini} we obtain \eqref{Fubini eq}, 
which finishes the proof of the theorem. 
\end{proof}
\begin{remark}
There is a Fubini theorem for stochastic 
integrals with respect to semimartingales 
in \cite{Pr2005}, see Theorem 65 in Chapter lV.  
Its integrability condition applied to our situation 
reads as 
\begin{equation}                                                      \label{ProtterFubini}
\int_0^T\int_Z\int_\Lambda
|f(t,z,\lambda)|^2\,m(d\lambda)\,\mu(dz)\,dt 
<\infty  \quad(a.s.), 
\end{equation}
which for finite measures $m$ is stronger than condition \eqref{Fubini cond}. 
\end{remark}
We also need a Fubini theorem for integrals against 
the Poisson random measure $\pi(dz,dt)$, 
that we formulate it as follows.
\begin{theorem}                                                                                \label{theorem Fubini2}
Let $g=g(\omega,t,z,\lambda)$ be a real-valued 
$\cP\otimes\cZ\otimes\cS$-measurable function on 
$\Omega\times[0,T]\times Z\times\Lambda$ such that
\begin{equation}                                                                \label{condition2}
\int_0^T\int_Z|g(t,z,\lambda)|\,\mu(dz)\,dt<\infty
\end{equation}
for each $\lambda\in\Lambda$ and $\omega\in\Omega$. 
Then there exists an $\cF\otimes \cB([0,T])\otimes\cS$-measurable 
function $G=G(t,\lambda)$ such that it is cadlag in $t\in[0,T]$ 
for each $(\omega,\lambda)\in\Omega\times\Lambda$, 
for each $\lambda\in\Lambda$ the process $(G(t,\lambda))_{t\in[0,T]}$ 
is locally integrable and $\cF_t$-adapted, and
$$
G(t,\lambda)=\int_0^t\int_Zg(s,z,\lambda)\,\pi(dz,ds)
$$
almost surely for all $t\in[0,T]$. Furthermore, if almost surely 
$$
\int_\Lambda\int_0^T\int_Z|g(t,z,\lambda)|\,\mu(dz)\,dt\,m(d\lambda)<\infty,
$$
then 
$$
\int_\Lambda G(t,\lambda)\,m(d\lambda)
=\int_0^t\int_Z\int_\Lambda g(s,z,\lambda)\,m(d\lambda)\,\pi(dz,ds)
$$
almost surely for all $t\in[0,T]$.
\end{theorem} 
\begin{proof}
One knows, see e.g. \cite{IW2011} that 
$$
E\int_0^T\int_{Z}|f_s(z)|\,\pi(dz, ds)=E\int_0^T\int_{Z}|f_s(z)|\,\mu(dz)\,ds
$$
for $\cP\otimes\cS$-measurable real-valued functions $f$ 
on $\Omega\times[0,T]\times Z$, 
when the right-hand side of the above equation is finite. 
Using this identity, we can prove this 
theorem by a straightforward modification of 
the proof of Theorem \ref{theorem Fubini1} 
above. 
\end{proof}

For $\sigma$-finite measure spaces $(\Lambda_i,\cS_i,\mu_i)$, 
a separable real Banach space $V$ 
and $p_i\in[1,\infty)$ for $i=1,2$ let $L_{p_1,p_2}$ 
denote the space of $V$-valued 
$\cS_1\otimes S_2$-measurable functions $f=f(x,y)$ 
of $(x,y)\in\Lambda_1\times\Lambda_2$,  
such that 
$$
\int_{\Lambda_1}\left(\int_{\Lambda_2}|f(x,y)|_V^{p_2}
\,\mu_2(dy)\right)^{p_1/p_2}\,\mu_1(dx)<\infty.  
$$
Assume that $(\Lambda_2,\cS_2,\mu_2)$ is separable, 
and let $L_{p_1}(L_{p_2}(V))$ 
denote the space of $\cS_1$-measurable functions $f$ 
mapping $\Lambda_1$ into the space 
$L_{p_2}(V)=L_{p_2}((\Lambda_2,\cS_2,\mu_2),V)$ 
equipped with the Borel $\sigma$-algebra,  
such that 
$$
\int_{S_1}|f(x)|^{p_1}_{L_{p_2}(V)}\,\mu_1(dx)<\infty.
$$
Then we have the following lemma.
\begin{lemma}                                           \label{lemma L}
The spaces $L_{p_1,p_2}$ and $L_{p_1}(L_{p_2})$ 
are the same in the sense that for each  
$f$ from $L_{p_1}(L_{p_2})$ there is $\bar f\in L_{p_1,p_2}$ 
such that for every $x\in\Lambda_1$ we have 
$\bar f(x,y)=f(x,y)$ for $\mu_2$-a.e. $y\in\Lambda_2$, 
and for each $g\in L_{p_1,p_2}$ 
there is $\tilde g\in L_{p_1}(L_{p_2})$ 
such that for $\mu_1$-a.e. $x\in \Lambda_1$ we have 
$g(x,y)=\tilde g(x,y)$ for all $y\in\Lambda_2$. 
\end{lemma}
\begin{proof}
Due to the separability of $(\Lambda_2,\cS_2,\mu_2)$ 
and $V$, there are countable subsets 
$\cS_0\subset\cS_2$ and $V_0\subset V$ such that the space $\cV$ of functions 
$g$ of the form 
$$
g(y)=\sum_{i=1}^N{\bf1}_{\Gamma_i}(y)v_i 
\quad 
\text{for $\Gamma_i\in \cS_0$, $\mu_2(\Gamma_i)<\infty$, $v_i\in V_0$, $N=1,2,...,$}
$$
is a countable dense subspace of $L_{p_2}(V)$.  
Hence for any $\cS_1$-measurable function 
$$
f:\Lambda_1\to L_{p_2}(V)
$$ 
there is a sequence $(f^n)_{n=1}^{\infty}$ of $\cV$-valued functions 
of the form 
$$
f^n(x)=\sum_{i=1}^{\infty}{\bf1}_{F^n_i}(x)g^n_i,
$$
such that  
$F^n_i\in\cS_1$, $F^n_i\cap F^n_j=\emptyset$ for $i\neq j$, $g^n_i\in\cV$  
and 
$$
|f^n(x)-f(x)|_{L_{p_2}(V)}<2^{-n-1}\quad \text{for 
all $x\in\Lambda_1$}  
$$
for $n\geq1$. Thus for each $x\in \Lambda_1$ for the set 
$$
A_n(x)=\{y\in S_2: |f^{n+1}(x,y)-f^n(x,y)|_{V}\geq n^{-2}\}\in\cS_2
$$ 
we have $\mu_2(A_n(x))\leq n^{p_2}2^{-p_2n}$,  
which, due to $\sum_{n=1}^{\infty}\mu_2(A_n(x))<\infty$, 
implies that for each $x\in\Lambda_1$ 
the sequence $(f^n(x,y))_{n=1}^{\infty}$ is convergent 
in $V$ for $\mu_2$-almost every $y\in\Lambda_2$. 
Define 
$$
B=\{(x,y)\in\Lambda_1\times\Lambda_2: 
(f^n(x,y))_{n=1}^{\infty}\,\,\text{is convergent in $V$}\}, 
$$ 
$$
\bar f(x,y)=\begin{cases}
\lim_{n\to\infty}f^n(x,y)\,\,\text{for $(x,y)\in B$}\\
0\in V \quad\qquad\qquad\text{for $(x,y)\notin B$}. 
\end{cases}
$$
Then $B\in\cS_1\otimes\cS_2$, and hence $\bar f$ 
is $\cS_1\otimes\cS_2$-measurable.  
Moreover,  $\bar f(x,y)=f(x,y)$ for $\mu_2$-almost every 
$y\in\Lambda_2$ for every $x\in\Lambda_1$. 
Assume now that $g\in L_{p_1,p_2}$. Then $|g(x,\cdot)|_{L_{p_2}(V)}$ is an 
$\cS_1$-measurable function of $x\in\Lambda_1$, 
with values in $[0,\infty]$. In particular, 
$$
A:=\{x\in\Lambda_1: |g(x,\cdot)|_{L_{p_2}(V)}<\infty\}\in \cS_1,  
$$
and $\mu_1(\Lambda_1\setminus A)=0$ by Fubini's theorem. 
For the function $\tilde g(x,y)={\bf1}_A(x)g(x,y)$ by Fubini's theorem we have 
$$
\{x\in\Lambda_1:|\tilde g(x)-e|_{L_{p_2}(V)}<R\}\in\cS_1
$$
for any $e\in L_{p_2}(V)$ and $R>0$. Consequently, 
$\tilde g$ is an $\cS_1$-measurable 
$L_{p_2}(V)$-valued function on $\Lambda_1$. 
In particular, $\tilde g\in L_{p_1}(L_{p_2})$, 
and clearly, for $\mu_1$-almost every $x\in\Lambda_1$ 
we have $\tilde g(x,y)=g(x,y)$ for every $y\in\Lambda_2$. 
\end{proof}
Recall that $\bL_p(L_p)$, $\bL_p(L_p(\ell_2))$ and $\bL_p(L_p(\cL_{p,2}))$ denote 
the spaces of predictable functions defined on $\Omega\times[0,T]$ and taking values in 
$L_p=L_p(\bR^d,\bR^M))$, $L_p(\ell_2)=L_p(\bR^d,\ell_2)$ and in 
$L_p(\cL_{p,2})=L_p(\bR^d,\cL_{p,2})$, 
respectively. For separable Banach spaces $B$ and numbers $p,q\in[1,\infty)$ 
the notations 
$$
L_p(\Omega\times[0,T]\times\bR^d, V)\quad\text{and}\quad  
L_{p,q}(\Omega\times[0,T]\times\bR^d\times Z, V) 
$$
mean the space of $\cP\otimes\cB(\bR^d)$-measurable functions 
$f:\Omega\times[0,T]\times\bR^d\to V$ and 
the space of $\cP\otimes\cB(\bR^d)\otimes\cZ$-measurable 
functions $g:\Omega\times[0,T]\times \bR^d\times Z\to V$, respectively, 
such that 
$$
E\int_0^T\int_{\bR^d}|f_t(x)|^p_V\,dx\,dt<\infty\quad 
E\int_0^T\int_{\bR^d}\left(\int_Z|g_t(x,z)|^q_V\,\mu(dz)\right)^{p/q}\,dx\,dt<\infty. 
$$
\begin{corollary}                                                    \label{lemma identification}
 The following identifications hold in the sense of Lemma \ref{lemma L}: 
$$
\bL_p(L_p)=L_p(\Omega\times[0,T]\times\bR^d,\bR^M),\quad
\bL_p(L_p(\ell_2))=L_p(\Omega\times[0,T]\times\bR^d,\ell_2(\bR^M))
$$
$$
\bL_p(L_p(\cL_{p,2}))=L_p(\Omega\times[0,T]\times\bR^d, \cL_{p,2})
$$
$$
=L_p(\Omega\times[0,T]\times\bR^d, \cL_{p})\cap
L_p(\Omega\times[0,T]\times\bR^d, \cL_{2})
$$ 
$$
=L_{p,p}(\Omega\times[0,T]\times\bR^d\times Z, \bR^M)\cap
L_{p,2}(\Omega\times[0,T]\times\bR^d\times Z, \bR^M).  
$$
\end{corollary}
\begin{proof}
By definition of intersection spaces 
$$
L_p(\Omega\times[0,T]\times\bR^d, \cL_{p,2})
=L_p(\Omega\times[0,T]\times\bR^d, \cL_{p})\cap
L_p(\Omega\times[0,T]\times\bR^d, \cL_{2})
$$ 
as vector spaces, and it is easy to see that their norms are equivalent. 
The other equalities can be obtained by repeated applications of Lemma \ref{lemma L}. 
\end{proof}

We conclude this section with a simple lemma, which 
plays a useful role in situations when we want 
to use Lebesgue's theorem on dominated convergence to 
pass to the limit in some expressions in the proof of the main 
theorems. 
\begin{lemma}                                                                           \label{lemma tool}
Let $(V, |\cdot|_V)$ be a real Banach space whose elements are 
real-valued functions on a set $\Lambda$ such 
that when $f\in V$ then $|f|$, the absolute value 
of $f$, belongs to $V$ as well, and the norms of $f$ and $|f|$ are the same. 
Assume that the pointwise limit of every increasing sequence 
of non-negative functions $f_n\in V$ belongs to $V$ if $\sup_n|f_n|_V<\infty$. 
Then for every convergent sequence $(g_n)_{n=1}^{\infty}$ in $(V,|\cdot|_V)$ 
there is a subsequence $(g_{n(k)})_{k=1}^{\infty}$ and an element 
$G$ from $V$ such that $|g_{n(k)}|\leq G$ for each $k$. 
\end{lemma}
\begin{proof}
If $(g_n)_{n=1}^{\infty}$ is a Cauchy sequence in $(V,|\cdot|_V)$ 
then there  is a strictly increasing sequence of positive integers 
$(n(k))_{k=1}^{\infty}$ such that $|g_{n(k+1)}-g_{n(k)}|_V\leq 2^{-k}$ for each $k\geq1$. 
Thus 
$$
G:=|g_{n(1)}|+\sum_{k=1}^{\infty}|g_{n(k+1)}-g_{n(k)}|\in V
\quad \text{and\,\, $|g_{n(k)}|\leq G$ for every $k\geq1$}. 
$$
\end{proof}

\mysection{Proof of the main results}                             \label{section proof}
We use ideas and methods from 
\cite{K2010}. To prove the existence of the process $\bar u$ with the 
stated properties in Theorem \ref{theorem1}, first we show  
that when $\varphi$ runs through $C^{\infty}_0$, then the integral processes of $(f,\varphi)$, $(g,\varphi)$ 
and $(h,\varphi)$ in equation \eqref{eq1} define appropriate $L_p$-valued 
integral processes of $f$, $g$ and $h$, respectively. 
To this end we introduce a class of functions $\cU_p$,   
the counterpart of the class  $\cU_p$  introduced in \cite{K2010}.  

Let $\mathcal{U}_p$ denote the set of $\bR^M$-valued functions 
$u=u_t(x)=u_t(\omega,x)$ on $\Omega\times [0,T]\times \mathbb{R}^d$ such that
\begin{enumerate}[(i)]
\item $u$ is $\mathcal{F}\otimes\mathcal{B}([0,T])\otimes\mathcal{B}(\mathbb{R}^d)$-measurable,
\item for each $x\in\mathbb{R}^d$, $u_t(x)$ is $\mathcal{F}_t$-adapted,
\item $u_t(x)$ is cadlag in $t\in[0,T]$ for each $(\omega,x)$,
\item $u_t(\omega,\cdot)$ as a function of 
$(\omega,t)$ is $L_p$-valued, $\mathcal{F}_t$-adapted and cadlag in $t$ 
for every $\omega\in \Omega$.
\end{enumerate}
The following two lemmas are obvious corollaries 
of Lemmas 4.3 and 4.4 in \cite{K2010}. 

\begin{lemma}                                         \label{lemma f}
Let $f$ be an $\bR^M$-valued function from $\bL_p$. 
Then there exists a function $m\in \cU_p$ 
such that for each $\varphi\in C_0^\infty$ almost surely
$$
(m_t,\varphi)=\int_0^t(f_s,\varphi)\,ds
$$
holds for all $t\in[0,T]$. Furthermore, we have
$$
E\int_{\bR^d}\sup_{t\leq T}|m_t(x)|^p\,dx\leq NT^{p-1}E\int_0^T|f_s|^p_{L_p}\,ds,
$$
with a constant $N=N(p,M)$. 
\end{lemma}

\begin{lemma}                                                                   \label{lemma g}
Let $g$ be an $\ell_2$-valued function from $\bL_p$. 
Then there exists a function $a\in\cU_p$ 
such that for each $\varphi\in C_0^\infty$ almost surely 
$$
(a_t,\varphi)=\sum_{r=1}^\infty \int_0^t(g_s^r,\varphi)\,dw_s^r
$$
holds for all $t\in[0,T]$. Furthermore, we have
$$
E\int_{\bR^d}\sup_{t\leq T}|a_t(x)|^p\,dx\leq NT^{(p-2)/2}E\int_0^T|g_s|^p_{L_p}\,ds,
$$
with a constant  $N=N(p,M)$.
\end{lemma}

\begin{lemma}                                                                        \label{lemma h}
Let $h\in\mathbb{L}_{p,2}$ for $p\geq2$. 
Then there exists a function $b\in\mathcal{U}_p$ 
such that for each real-valued $\varphi\in L_q(\bR^d)$ 
with $q=p/(p-1)$, almost surely 
\begin{equation}                                                                   \label{weak3}
(b_t,\varphi)=\int_0^t\int_Z(h_s,\varphi)\,\tilde{\pi}(dz,ds)
\end{equation}
for all $t\in[0,T]$,   and 
\begin{equation}                                                                   \label{weaksup}
E\sup_{t\leq T}|(b_t,\varphi)|
\leq 3T^{(p-2)/(2p)}|\varphi|_{L_q}\left(E\int_0^T|h_t|_{L_p(\cL_2)}^p\,dt\right)^{1/p}. 
\end{equation}
Furthermore
\begin{equation}                                                                       \label{Lp}
E\int_{\bR^d}\sup_{t\leq T}|b_t(x)|^p\,dx
\leq NE\int_0^T|h_t|^p_{L_p(\cL_p)}\,dt+NT^{(p-2)/2}E\int_0^T|h_t|^p_{L_p(\cL_2)}\,dt
\leq N'|h|^p_{\bL_{p,2}}
\end{equation}
with constants $N=N(p,M)$ and $N'=N'(p,M,T)$.
\end{lemma}
\begin{proof}
Let $\cH$ denote the set of functions $h$ of the form 
$$
h_t(z,x)=\sum_{i=1}^k\varphi_i(x)c_i{\bf1}_{(s_i,t_i]}(t){\bf1}_{U_i}(z)
$$
for integers $k\geq1$, functions $\varphi_i\in C^{\infty}_0(\bR^d)$, 
time points $0\leq s_i\leq t_i$, 
$\cF_{s_i}$-measurable bounded random vectors $c_i$  
and sets $U_i\in \cZ$ such that $\mu(U_i)<\infty$. 
For this function $h$ define $b$ by 
$$
b_t(x)=\sum_i \varphi_i(x)c_i(\tilde \pi_{t_{i}\wedge t}(U_i)-\tilde \pi_{s_{i\wedge t}}(U_i)), 
\quad t\in[0,T], \,x\in\bR^d.  
$$
Clearly, $b\in \cU_p$,  for  every  $\varphi\in L_q(\bR^d,\bR)$
$$
(b_t,\varphi)=\sum_{i=1}^k(\varphi_i,\varphi)
c_i(\tilde \pi_{t_{i}\wedge t}(U_i)-\tilde \pi_{s_{i\wedge t}}(U_i))
=\int_0^t\int_{\cZ}(h_s(z),\varphi)\,\tilde\pi(dz,ds)  
$$  
almost surely for all $t\in[0,T]$, and for each $x\in\bR^d$ 
\begin{equation}                                          \label{x}
b_t(x)=\int_0^t\int_Zh_s(x,z)\,\tilde\pi(dz,ds)\quad\text{almost surely for all $t$.}
\end{equation}
  By the Davis, Minkowski and H\"older inequalities,
$$
E\sup_{t\leq T}|(b_t,\varphi)|
\leq 3E\left(\int_0^T\int_Z|(h_s(z),\varphi)|^2\,\mu(dz)\,ds\right)^{1/2}
\leq 3E\left(\int_0^T(|h_s|_{\cL_2},|\varphi|)^2\,ds\right)^{1/2}
$$
$$
\leq 3E\left(\int_0^T(|h_s|^2_{L_p(\cL_2)}|\varphi|_{L_q}^2\,ds\right)^{1/2}
\leq 3|\varphi|_{L_q}\left(\int_0^TE|h_s|^2_{L_p(\cL_2)}\,ds\right)^{1/2}
$$
$$
\leq 3|\varphi|_{L_q}\left(\int_0^T(E|h_s|^p_{L_p(\cL_2)})^{2/p}\,ds\right)^{1/2}
\leq 3|\varphi|_{L_q}T^{(p-2)/(2p)}\left(\int_0^TE|h_s|^p_{L_p(\cL_2)}\,ds\right)^{1/p}, 
$$
which proves \eqref{weaksup} when $h\in\cH$. 
From \eqref{x} by the Burkholder-Davis-Gundy inequality 
for Poisson  martingale measures, 
see, e.g. \cite{MP3}, for $p\geq2$ for each $x\in\bR^d$ we have 
$$
E\sup_{t\in[0,T]}|b_t(x)|^p
=E\sup_{t\leq T}\left(\int_0^t\int_Zh_s(x,z)\,\tilde\pi(dz,ds)\right)^p
$$
$$
\leq NE\int_0^T\int_Z|h_s(x,z)|^p\,\mu(dz)\,ds+
 NE\left(\int_0^T\int_Z|h_s(x,z)|^2\,\mu(dz)\,ds\right)^{p/2}  
$$
with $N=N(p,M)$.  Hence by Jensen's inequality and integrating 
over $\bR^d$ we get \eqref{Lp} for $h\in\cH$. 
It is not difficult to see that $\cH$ is dense in $\bL_{p,2}$. Thus for 
$h\in\bL_{p,2}$ there is a sequence $h^n\in\cH$ and $b^n\in\cU_p$, 
such that $h^n\to h$ in $h\in\bL_{p,2}$, and \eqref{weak3} and \eqref{Lp} 
hold with $b^n$ and $h^n$ in place of $b$ and $h$,  respectively. 
Therefore we can find a subsequence $h^{n(k)}$ and $b^{n(k)}$ such that 
$$
E\int_{\bR^d}\sup_{t\leq T}|b^{n(k+1)}_t(x)-b^{n(k)}_t(x)|^p\,dx
$$
$$
\leq 
N(E\int_0^T|h_t^{n(k+1)}-h_t^{n(k)}|_{L_p(\cL_2)}^p\,dt
+E\int_0^T|h_t^{n(k+1)}-h_t^{n(k)}|_{L_p(\cL_p)}^p\,dt)
\leq \frac{N}{2^{kp}}. 
$$
Hence there is a set $\Theta\in\cF\otimes\cB(\bR^d)$ of full measure 
such that for $k\to\infty$ the 
sequence $b_t^{n(k)}(x)$ converges for $(t,\omega,x)\in[0,T]\times\Theta$, 
uniformly in $t\in[0,T]$.  Define 
$$
\Gamma=\{x\in\bR^d: P((\omega,x)\in\Theta)=1\}\quad 
\text{and
\quad  
$\bar\Theta=\Theta\cap(\Omega\times \Gamma)$}. 
$$ 
By Fubini's theorem $\Gamma\in\cB(\bR^d)$, and it is of full measure. 
Hence $\bar\Theta\in\cF\otimes\cB(\bR^d)$, and it is of full measure. 
If $x\in\Gamma$ then $\bar\Theta_{x}:=\{\omega\in\Omega:(\omega,x)\in\bar\Theta\}
=\{\omega\in\Omega:(\omega,x)\in\Theta\}=:\Theta_{x}$, i.e., 
$P(\bar\Theta_{x})=P(\Theta_{x})=1$, which implies
$\bar\Theta_{x}\in\cF_0$, since $\cF_0$ is complete. 
If $x\notin\Gamma$ then $\bar\Theta_{x}=\emptyset\in\cF_0$. 
Thus $\bar b^{n(k)}:=b^{n(k)}{\bf1}_{\bar\Theta}$ is 
$\cF\otimes\cB([0,T])\otimes\cB(\bR^d)$-measurable 
and $\bar b^{n(k)}(t,x)$ is $\cF_t$-measurable for each 
$(t,x)\in[0,T]\times\bR^d$. Consequently, 
$b=\lim_{k\to\infty}\bar b^{n(k)}$ has these measurability properties as well. 
Since for every $(\omega,x)\in\Omega\times\bR^d$ the functions 
$\bar b^{n(k)}$ are cadlag and converge  to $b$, uniformly 
in $t\in[0,T]$, the limit $b$ is a cadlag function of $t\in[0,T]$ 
for every $(\omega,x)\in\Omega\times\bR^d$. 
Thus $b$ satisfies the conditions (i), (ii) and 
(iii) in the definitions of $\cU_p$. Letting $k\to\infty$ in 
$$
E\int_{\bR^d}\sup_{t\leq T}|\bar b^{n(k)}_t(x)|^p\,dx
\leq N(E\int_0^T|h_t^{n(k)}|^p_{L_p(\cL_p)}\,dt+T^{(p-2)/2}
E\int_0^T|h_t^{n(k)}|^p_{L_p(\cL_2)})\,dt 
$$
and 
$$
E\sup_{t\leq T}|(\bar b^{n(k)}_t,\varphi)|\leq 3T^{(p-2)/(2p)}|\varphi|_{L_q}
(E\int_0^T|h_t^{n(k)}|_{L_p(\cL_2)}^p\,dt)^{1/p} 
$$
by Fatou's lemma we get \eqref{Lp} and \eqref{weaksup}.   
Letting $k\to\infty$ in 
$$
E\sup_{t\leq T}|\bar b^{n(k)}_t-\bar b^{n(l)}_t|^p_{L_p}
\leq N E\int_0^T|h_t^{n(k)}-h_t^{n(l)}|^p_{L_p(\cL_p)}\,dt
+|h_t^{n(k)}-h_t^{n(l)}|^p_{L_p(\cL_2)}\,dt
$$
by Fatou's lemma we have 
$$
E\sup_{t\leq T}|b_t-\bar b_t^{n(l)}|^p_{L_p}
\leq N E \int_0^T |h_t-h_t^{n(l)}|_{L_p(\cL_p)}^p+
|h_t-h_t^{n(l)}|_{L_p(\cL_2)}^p\,dt,
$$
which converges to zero as $l\to\infty$. Thus there is 
$\Omega^{'}\subset\Omega$ of full 
probability such that $({\bf1}_{\Omega^{'}}b_t)_{t\in[0,T]}$ is an $L_p$-valued 
$\cF_t$-adapted cadlag process. For $\varphi\in L_q(\bR^d)$ using the Davis 
inequality, then Minkowski's and H\"older's inequalities we have 
$$
E\sup_{t\in[0,T]}|(\bar b^{n(k+1)}_t,\varphi)-(\bar b^{n(k)}_t,\varphi)|
$$
$$
\leq N'|\varphi|_{L_q}\left(E\int_0^T|h^{n(k+1)}-h^{n(k)}|^p_{L_p(\cL_2)}\,dt\right)^{1/p}
\leq N{''}|\varphi|_{L_q}2^{-k}
$$
with constants $N'=N'(p,T)$ and $N{''}=N{''}(p,T)$. Hence we can see that 
letting $k\to\infty$ in 
$$
(\bar b_t^{n(k)},\varphi)=\int_0^t\int_Z(h^{n(k)}_s(z),\varphi)\,\tilde\pi(dz,ds), 
$$
both sides converge almost surely, uniformly in $t\in[0,T]$,  
and for each $\varphi\in L_q$ we get that there is 
$\Omega_{\varphi}\subset\Omega$ 
of full probability 
such that for $\omega\in\Omega_{\varphi}$ 
$$
(b_t,\varphi)=\int_0^t\int_Z(h_s(z),\varphi)\,\tilde\pi(dz,ds)
$$
for all $t\in[0,T]$, which completes the proof of the lemma.  
\end{proof}

Following \cite{K2010} we obtain It\^o's formula \eqref{Ito1}  by mollifying 
$\bar u$ in $x\in\bR^d$ and applying   
It\^o's formula \eqref{Itop}. 
To this end we take a nonnegative kernel 
$k\in C_0^{\infty}$ with unit integral, and for $\varepsilon\in(0,1)$ 
and for locally integrable functions $v$ of $x\in\bR^d$ we use the notation  
$v^{(\varepsilon)}$ for the mollifications of $v$, 
\begin{equation}                                          \label{mollification def}
v^{(\varepsilon)}(x)=\int_{\bR^d}v(x-y)k_{\varepsilon}(y)\,dy, \quad x\in\bR^d, 
\end{equation}
where $k_{\varepsilon}(y)=\varepsilon^{-d}k(y/\varepsilon)$ for $y\in\bR^d$. 
Note that if $v=v(x)$ is a locally Bochner-integrable function on $\bR^d$,  
taking values in a Banach space, the mollification of $v$ is defined 
as \eqref{mollification def} in the sense of Bochner integral.

We will make use of well-known smoothness properties of mollifications 
and the following well-known lemma. 

\begin{lemma}                                                               \label{lemma epsilon}
Let $V$ be a separable Banach space, 
and let $f=f(x)$ be a $V$-valued function of $x\in\bR^d$ such that 
$f\in L_p(V)=L_p(\bR^d,V)$ for some $p\geq 1$. Then 
$$
|f^{(\varepsilon)}|_{L_p(V)}\leq |f|_{L_p(V)}\quad\text{for every $\varepsilon>0$}, 
\quad  
\text{and} 
\quad \lim_{\varepsilon\to 0}|f^{(\varepsilon)}-f|_{L_p(V)}=0.
$$
\end{lemma}
\begin{proof} By the properties of Bochner integrals, Jensen's inequality 
and Fubini's theorem 
$$
|f^{(\varepsilon)}|^p_{L_p(V)}
=\int_{\bR^d}\left|\int_{\bR^d}f(y)k_\varepsilon(x-y)dy\right|^p_V\,dx
$$
$$
\leq \int_{\bR^d}\int_{\bR^d}|f(y)|^p_V k_\varepsilon(x-y)\,dy\,dx=|f|^p_{L_p(V)}.
$$
Since $V$ is separable, it has a countable dense subset $V_0$. 
Denote by $\cH\subset L_p(V)$ 
the space of functions $h$ of the form 
$$
h(x)=\sum_{i=1}^kv_i\varphi_i(x)
$$ 
for some integer $k\geq1$, $v_i\in V_0$ 
and 
continuous real functions $\varphi_i$ on $\bR^d$ with compact support. 
Then for such an $h$ we have 
$$
|h^{(\varepsilon)}-h|_{L_p(V)}
\leq \sum_{i=1}^{k}
|\varphi_i^{(\varepsilon)}-\varphi_i|_{L_p}
|v_i|_V\to0
\quad\text{as $\varepsilon\to0$},  
$$
where $L_p=L_p(\bR^d,\bR)$. 
For $f\in L_p(V)$ and $h\in\cH$ we have 
$$
|f-f^{(\varepsilon)}|_{L_p(V)}\leq |f-h|_{L_p(V)}
+|h-h^{(\varepsilon)}|_{L_p(V)}+|(f-h)^{(\varepsilon)}|_{L_p(V)}
\leq 2|f-h|_{L_p(V)}+|h-h^{(\varepsilon)}|_{L_p(V)}.
$$
Letting here $\varepsilon\to0$ for each $f\in L_p(V)$ we obtain 
\begin{equation*}                                                                           \label{fh}
\limsup_{\varepsilon\to0}|f-f^{(\varepsilon)}|_{L_p(V)}
\leq 2|f-h|_{L_p(V)}\quad \text{for all $h\in\cH$}. 
\end{equation*}
Since $\cH$ is dense in  $L_p(V)$, 
we can choose $h\in\cH$ to make $|f-h|_{L_p(V)}$ arbitrarily small,   
which proves $\lim_{\varepsilon\to0}|f-f^{(\varepsilon)}|_{L_p(V)}=0$. 
\end{proof}
\begin{proof}[Proof of Theorem \ref{theorem1}]
By Lemmas \ref{lemma f}, \ref{lemma g} and \ref{lemma h} there exist $a=(a^i)$ 
and $b=(b^i)$ and $m=(m^i)$ in $\mathcal{U}_p$ 
such that for each $\varphi\in C_0^\infty$ almost surely
$$
(a_t^i,\varphi)=\int_0^t(f^i_s,\varphi)ds,
\quad 
(b_t^i,\varphi)=\int_0^t(g^{ir}_s,\varphi)dw_s^r
$$
and
$$
(m^i_t,\varphi)=\int_0^t\int_Z(h^i_s,\varphi)\,\tilde{\pi}(dz,ds)
$$
for all $t\in [0,T]$ and $i=1,...,M$. Thus $a+b+m$ is an 
$L_p$-valued adapted cadlag process such that 
for $\bar u_t:=\psi+a_t+b_t+m_t$ we have 
$(\bar u_t,\varphi)=
(u_t,\varphi)$ for each $\varphi\in C_0^{\infty}$ 
for $P\otimes dt$ almost every 
$(\omega,t)\in\Omega\times[0,T]$. 
Hence, by taking a countable set 
$\Phi\subset C_0^{\infty}$ such that $\Phi$ is dense in $L_q$, 
we get that  $\bar u=u$ 
for $P\otimes dt$ almost everywhere 
as $L_p$-valued functions. Moreover, for each 
$\varphi\in C_0^{\infty}$
\begin{equation}                                                          \label{equ1}
(\bar u^i_t,\varphi)
=(\psi,\varphi)
+\int_0^t(f^i_s,\varphi)\,ds
+\int_0^t(g_s^{ir},\varphi)\,dw_s^r
+\int_0^t\int_Z(h^i_s(z),\varphi)\,\tilde{\pi}(dz,ds)
\end{equation}
almost surely for all $t\in[0,T]$, $i=1,2,...,M$, 
since both sides we have cadlag 
processes. 
To ease notation we will denote $\bar u$ also by $u$ in the sequel. 
 
By the estimates of Lemmas 
\ref{lemma f}, \ref{lemma g} and \ref{lemma h},
$$
E\int_{\mathbb{R}^d}\sup_{t\leq T}|u_t(x)|^p\,dx
$$
\begin{equation}                                            \label{u sup estimate}
\leq N\left( E|\psi|^p_{L_p}+ |f|^p_{\bL_p}
+|g_s|^p_{\bL_p}
+|h|^p_{\bL_{p,2}} \right)<\infty, 
\end{equation}
where $N=N(p,M, T)$ is a constant.  
Substituting $k_{\varepsilon}(x-\cdot)=\varepsilon^{-d}k((x-\cdot)/\varepsilon)$ 
in place of $\varphi$  in equation \eqref{equ1}, 
for $\varepsilon>0$ and $x\in\bR^d$ 
we have 
$$
u_t^{(\varepsilon)i}(x)=\psi^{(\varepsilon)i}(x)
+\int_0^tf^{(\varepsilon)i}_s(x)\,ds
+\int_0^tg^{(\varepsilon)ir}_s(x)\,dw^r_s
+\int_0^t\int_Zh_s^{(\varepsilon)i}(x)\,\tilde{\pi}(dz,ds)
$$
almost surely for all $t\in[0,T]$ for $i=1,2,...,M$. 
By virtue of Lemma \ref{lemma epsilon} we have
\begin{equation}                                                                  \label{h epsilon conv}
\lim_{\varepsilon\to0}|h^{(\varepsilon)}-h|^p_{\bL_{p,2}}=0
\end{equation}
and     
\begin{equation}                                                                        \label{f g epsilon conv}
\lim_{\varepsilon\to0}(|
f^{(\varepsilon)}-f|_{\bL_p}+|g^{(\varepsilon)}-g|_{\bL_p})=0.  
\end{equation}
Then using \eqref{h epsilon conv}, \eqref{f g epsilon conv} 
and estimate \eqref{u sup estimate} 
with $u^{(\varepsilon)}-u$ in place of $u$, we have
\begin{equation}                                                                              \label{u epsilon conv}
\lim_{\varepsilon\to0}E\,\sup_{t\leq T}
|u_t^{(\varepsilon)}-u_t|^p_{L_p}=0.
\end{equation}
By Minkowski's and H\"older's inequalities, for $\varepsilon>0$ 
for each $x\in \mathbb{R}^d$, $s\in[0,T]$ and $\omega\in\Omega$
\begin{equation}                                                                                \label{he}
|h^{(\varepsilon)}_s(x)|_{\cL_2}
\leq \int_{\mathbb{R}^d}|h_s(y)|_{\cL_2}k_\varepsilon(x-y)\,dy
\leq N_{\varepsilon}|h_s|_{L_p(\cL_2)}   
\end{equation}
with 
$N_{\varepsilon}=|k_{\varepsilon}|_{L_{p/(p-1)}}<\infty$.   
Similarly,  for every $\varepsilon>0$
$$
|f^{(\varepsilon)}_s(x)|\leq  N_{\varepsilon}|f_s|_{L_p},
\quad 
|g_s^{(\varepsilon)}(x)|_{\ell_2}
\leq N_{\varepsilon}|g_s|_{L_p}.  
$$
Hence 
$$
\int_0^T|f^{(\varepsilon)}_s(x)|+|g_s^{(\varepsilon)}(x)|^2_{\ell_2}
+|h^{(\varepsilon)}_s(x)|^2_{\cL_2}<\infty\,\,\text{(a.s.)}. 
$$
Thus we can apply Theorem \ref{theorem0} on It\^o's formula  
to $|u_t^{(\varepsilon)}(x)|^p$ for each $x\in \bR^d$ 
to get 
$$
|u^{(\varepsilon)}_t(x)|^p
=  |\psi^{(\varepsilon)}(x)|^p
+\int_0^t
p|u^{(\varepsilon)}_{s-}(x)|^{p-2}u^{(\varepsilon)i}_{s-}(x)g^{(\varepsilon)ir}_s(x)\,dw^r_s 
$$
$$
+\int_0^t
p|u_{s-}^{(\varepsilon)}(x)|^{p-2}u_{s-}^{(\varepsilon)i}f_s^{(\varepsilon)i}(x) \,ds 
$$
$$             
+\tfrac{p}{2}\int_0^t\big((p-2)
|u_{s-}^{(\varepsilon)}(x)|^{p-4}|u_{s-}^{(\varepsilon)i}(x)g^{(\varepsilon)i\cdot}_s(x)|^2_{l_2}
+
|u_{s-}^{(\varepsilon)}(x)|^{p-2}|g_s^{(\varepsilon)}(x)|_{\ell_2}^2  \big)\,ds      
$$
\begin{equation}                                                                    \label{Itox}
+\int_0^t
\int_Z p|u_{s-}^{(\varepsilon)}(x)|^{p-2}u_{s-}^{(\varepsilon)i}(x)h_s^{(\varepsilon)i}(x)
\,\tilde{\pi}(dz,ds)
+\int_0^t\int_Z J^{h_s^{(\varepsilon)}(x,z)}|u_{s-}^{(\varepsilon)}(x)|^p \,\pi(dz,ds),        
\end{equation}
where the notation 
$$
J^a|v|^p=|v+a|^p-|v|^p-a^iD_{i}|v|^p=|v+a|^p-|v|^p-pa^i|v|^{p-2}v^i
$$
is used for vectors $a=(a^1,...,a^M)$ and $(v^1,...,v^M)\in\bR^M$. 
In order to integrate both sides of  \eqref{Itox} against $dx$ 
over $\bR^d$ and apply deterministic and stochastic 
Fubini theorems, we  are going to check 
that almost surely 
\begin{equation*}                                               \label{conditionA1}
A_1(x):=
\int_0^T\int_Z
|J^{h^{(\varepsilon)}_s}|u_{s-}^{(\varepsilon)}(x)|^p|\,\mu(dz)\,ds<\infty 
\quad\text{for all $x\in\bR^d$}, 
\end{equation*}
\begin{equation*}                                                   \label{conditionB1}
B_1:=\int_0^T\int_Z\int_{\mathbb{R}^d} 
|J^{h^{(\varepsilon)}_s}|u_{s-}^{(\varepsilon)}|^p|\,dx\,\mu(dz)\,ds<\infty, 
\end{equation*}
\begin{equation*}                                               \label{conditionA2}
A_2(x):=
 \int_0^T\int_Z 
|u_{s-}^{(\varepsilon)}(x)|^{2p-2}|h^{(\varepsilon)}_s(x)|^2\,\mu(dz)\,ds<\infty
\quad\text{for all $x\in\bR^d$}, 
\end{equation*}
\begin{equation*}                                               \label{conditionB2}
B_2:=\int_{\mathbb{R}^d}
\left( \int_0^T\int_Z 
|u_{s-}^{(\varepsilon)}|^{2p-2}|h^{(\varepsilon)}_s|^2\,\mu(dz)\,ds\right)^{1/2}\,dx<\infty, 
\end{equation*}    
\begin{equation*}                                               \label{conditionA3}
A_3(x):=\int_0^T
|u^{(\varepsilon)}_{s-}|^{2p-4}|u^{(\varepsilon)i}_{s-}(x)g^{(\varepsilon)i\cdot}_s(x)|^2_{l_2}\,ds
<\infty
\quad\text{for all $x\in\bR^d$}, 
\end{equation*}
\begin{equation*}                                               \label{conditionB3}
B_3:=\int_{\mathbb{R}^d}
\left(\int_0^T
|u^{(\varepsilon)}_{s-}|^{2p-4}
|u^{(\varepsilon)i}_{s-}g^{(\varepsilon)i\cdot}_s|^2_{l_2}\,ds
\right)^{1/2}\,dx<\infty  
\end{equation*}
and 
\begin{equation*}                                               \label{conditionC}
C:=\int_0^T\int_{\bR^d}
|u_s^{(\varepsilon)}(x)|^{p-1}|f_s^{(\varepsilon)}(x)|\,dx \,ds<\infty.  
\end{equation*}
To this end notice first that for $a,v\in\bR^M$ by Taylor's formula 
\begin{equation}                                                            \label{Jestimate}
|J^a|v|^p|\leq N(|v|^{p-2}|a|^2+|a|^p)
\end{equation}
with a constant $N=N(p)$. 
Using this and Young's inequality,    
by  \eqref{he} combined with 
$$
|u^{(\varepsilon)}_s(x)| \leq N_{\varepsilon}|u_s|_{L_p}
$$
 we get that almost surely  
$$
A_1(x)  \leq N\int_0^T
\int_Z
\left( 
|u_{s-}^{(\varepsilon)}(x)|^{p-2}
|h_s^{(\varepsilon)}(x,z)|^2+|h_s^{(\varepsilon)}(x,z)|^p 
\right)
\,\mu(dz)\,ds
$$
$$
\leq \tfrac{p-2}{p}NN^p_{\varepsilon}\int_0^T|u_{s-}|_{L_p}^p\,ds
+\tfrac{2}{p}NN^p_{\varepsilon}\int_0^T|h_s|_{L_p(\cL_2)}^p\,ds
+NN^p_{\varepsilon} \int_0^T |h_s|_{L_p(\cL_p)}^p\,ds< \infty  
$$
for all $x\in\bR^d$. 
By \eqref{Jestimate}, Young's inequality   
and Lemma \ref{lemma epsilon} 
 we have
$$
B_1  \leq N\int_0^T \int_{\mathbb{R}^d}
\int_Z
\left( 
|u_{s-}^{(\varepsilon)}|^{p-2}|h_s^{(\varepsilon)}|^2+|h_s^{(\varepsilon)}|^p 
\right)
\,\mu(dz)\,dx\,ds
$$
$$
\leq \tfrac{p-2}{p}N\int_0^T|u_{s-}|_{L_p}^p\,ds
+\tfrac{2}{p}N \int_0^T|h_s|_{L_p(\cL_2)}^p\,ds
+N \int_0^T |h_s|_{L_p(\cL_p)}^p\,ds< \infty\,\,\text{(a.s.)}.  
$$
Using \eqref{he} and 
$$
|u^{(\varepsilon)}_{s-}(x)|\leq N_{\varepsilon}|u_{s-}|_{L_p}
\quad
\text{for $s\in(0,T]$, $x\in\bR^d$, $\omega\in\Omega$}, 
$$
by \eqref{u sup estimate} we get that almost surely 
\begin{align*}
A_2(x) & \leq\sup_{t\leq T}|u_{t-}^{(\varepsilon)}(x)|^{2p-2}
\int_0^T|h_s^{(\varepsilon)}(x)|^2_{\cL_2}\,ds\\
& 
\leq N_{\varepsilon}^{2p} T^{(p-2)/p}\sup_{t\leq T}|u_{t}|^{2p-2}_{L_p}
\left(\int_0^T|h_s|_{L_p(\cL_2)}^p\,ds \right)^{2/p}<\infty\quad
\text{for all $x\in\bR^d$}.
\end{align*}
 By Young's and H\"older's inequalities and by \eqref{u sup estimate} 
\begin{align*}
B_2 & \leq q^{-1} \int_{\mathbb{R}^d}\sup_{t\leq T}|u_t^{(\varepsilon)}|^p\,dx
+p^{-1}
\int_{\mathbb{R}^d}
\left( \int_0^T|h_s^{(\varepsilon)}|^2_{\cL_2}\,ds \right)^{p/2}\,dx\\
& 
\leq q^{-1} \int_{\mathbb{R}^d}(\sup_{t\leq T}|u_t|^p)^{(\varepsilon)}\,dx+
T^{(p-2)/2}p^{-1}
\int_0^T|h_s^{(\varepsilon)}|_{L_p(\cL_2)}^p\,ds              \\
&
\leq q^{-1}
\int_{\mathbb{R}^d}\sup_{t\leq T}|u_t|^p\,dx
+
T^{(p-2)/2}p^{-1}\int_0^T|h_s|_{L_p(\cL_2)}^p\,ds 
<\infty \,\,\text{(a.s.)}  
\end{align*}
with  $q=p/(p-1)$. 
Similarly, for almost every $\omega\in\Omega$ 
$$
A_3(x)\leq N_{\varepsilon}^{2p} T^{(p-2)/p}\sup_{t\leq T}|u_{t}|^{2p-2}_{L_p}
\left(\int_0^T|g_s|_{L_p(\ell_2)}^p\,ds \right)^{2/p}<\infty\quad
\text{for all $x\in\bR^d$}, 
$$
$$
B_3\leq q^{-1}
\int_{\mathbb{R}^d}\sup_{t\leq T}|u_t|^p\,dx
+
T^{(p-2)/2}p^{-1}\int_0^T|g_t|^p_{L_p(\ell_2)}\,dt<\infty,  
$$
and 
$$
C\leq q^{-1}
\int_{\mathbb{R}^d}\sup_{t\leq T}|u_t|^p\,dx
+
T^{p-1}p^{-1}\int_0^T|f_t|^p_{L_p}\,dt<\infty. 
$$
Note that $u^{(\varepsilon)}_{t-}(x)$ is left continuous in $t$, it is continuous in $x$,  
and it is $\cF_t$-measurable for every $(t,x)$. 
Therefore $u_{t-}(x)$ is a $\cP\otimes\cB(\bR^d)$-measurable mapping of 
$(\omega,t,x)\in\Omega\times[0,T]\times\bR^d$, 
and hence it is easy to show that the integrands in \eqref{Itox} are also 
$\cP\otimes\cB(\bR^d)$-measurable functions of $(\omega,t,x)$. 

Thus integrating \eqref{Itox} over $\mathbb{R}^d$ we can use the deterministic 
Fubini theorem together with the stochastic Fubini theorems, Lemma 2.6 
from \cite{K2011} and Theorems  \ref{theorem Fubini1} and 
\ref{theorem Fubini2} above,  to get  
$$
|u^{(\varepsilon)}_{t}|_{L_p}^p= 
 |\psi^{(\varepsilon)}|_{L_p}^p
+\int_0^t\int_{\bR^d}
p|u^{(\varepsilon)}_s|^{p-2}u^{(\varepsilon)i}_sg^{(\varepsilon)ir}_s\,dx\,dw^r_s
$$
$$
+\tfrac{p}{2}\int_0^t\int_{\bR^d}
2|u_s^{(\varepsilon)}|^{p-2}u_s^{(\varepsilon)i}f_s^{(\varepsilon)i}               
+
(p-2)
|u_s^{(\varepsilon)}|^{p-4}|u_s^{(\varepsilon)i}g^{(\varepsilon)i\cdot}_s|_{l_2}^2
+|u_s^{(\varepsilon)}|^{p-2}|g_s^{(\varepsilon)}|_{l_2}^2
\,dx\,ds                                                                                                                    
$$
\begin{equation}                                                                    \label{Itoe}
+\int_0^t\int_Z\int_{\bR^d} 
p|u_{s-}^{(\varepsilon)}|^{p-2}u_{s-}^{(\varepsilon)i}h_s^{(\varepsilon)i}\,dx\,\tilde{\pi}(dz,ds)
+\int_0^t\int_Z\int_{\bR^d} 
J^{h^{(\varepsilon)}_s}|u_{s-}^{(\varepsilon)}|^p\,dx \,\pi(dz,ds)                                      
\end{equation}
almost surely for all $t\in[0,T]$. 
In order to take $\varepsilon\to0$ here, 
we need to prove
$$
A_\varepsilon:=
\int_0^T
\int_Z
\left( 
\int_{\mathbb{R}^d} |u_{s-}^{(\varepsilon)}|^{p-2}u_{s-}^{(\varepsilon)i}h_s^{(\varepsilon)i}
-|u_{s-}|^{p-2}u_{s-}^ih_s^i\,dx 
\right)^2\,\mu(dz)\,ds\to 0, 
$$
$$
B_\varepsilon:
=\int_0^T\int_Z\int_{\mathbb{R}^d}
\big|J^{h^{(\varepsilon)}_s}|u_{s-}^{(\varepsilon)}|^p-J^{h_s}|u_{s-}|^p\big| \,dx\,\pi(dz,ds)\to0  
$$
and
$$
C_{\varepsilon}:=\int_0^T
\left|
\int_{\mathbb{R}^d}
|u^{(\varepsilon)}_s|^{p-2}u^{(\varepsilon)i}_sg^{(\varepsilon)i\cdot}_s
-|u_s|^{p-2}u_s^ig_s^{i\cdot}\,dx\right|^2_{l_2}\,ds\to0
$$
in probability as $\varepsilon\rightarrow 0$.
To this end notice that $A_\varepsilon \leq A^1_\varepsilon +A^2_\varepsilon$ 
with 
$$
A_\varepsilon^1:
=\int_0^T\int_Z\left(\int_{\mathbb{R^d}}|u_{s-}^{(\varepsilon)}|^{p-1}
|h_s^{(\varepsilon)}-h_s|\,dx  \right)^2\,\mu(dz)\,ds
$$
and 
$$
A^2_\varepsilon:=
\int_0^T\int_Z\left(\int_{\mathbb{R}^d}
(|u_{s-}^{(\varepsilon)}|^{p-2}u_{s-}^{(\varepsilon)i}-|u_{s-}|^{p-2}u_{s-}^i)
h^i_s\,dx\right)^2\,\mu(dz)\,ds.
$$
By Minkowski's and H\"older inequalities 
\begin{align*}
A^1_\varepsilon 
& \leq \int_0^T
\left( \int_{\mathbb{R}^d}
\left(\int_Z
|u_{s-}^{(\varepsilon)}|^{2p-2}|h^{(\varepsilon)}_{s}-h_s|^2\,\mu(dz)
\right)^{1/2}\,dx 
\right)^2\,ds                                                                                      \\
& 
\leq 
\int_0^T|u_{s-}|^{2p-2}_{L_p}
\left(\int_{\mathbb{R}^d}|h^{(\varepsilon)}-h_s|^p_{\cL_2}\,dx\right)^{2/p}\,ds \\
& 
\leq \sup_{t\leq T}|u_{t}|^{2p-2}_{L_p}
\int_0^T|h_s^{(\varepsilon)}-h_s|^p_{L_p(\cL_2)}\,ds\rightarrow 0
\end{align*}
almost surely as $\varepsilon\rightarrow 0$,  and
\begin{align}
A^2_\varepsilon 
& \leq \int_0^T
\left(\int_{\mathbb{R}^d} 
\left( \int_Z
||u_{s-}^{(\varepsilon)}|^{p-2}u_{s-}^{(\varepsilon)}
-|u_{s-}|^{p-2}u_{s-}|^2|h_s|^2
\,\mu(dz) 
\right)^{1/2}\,dx
\right)^2\,ds                                               \nonumber\\
&
\leq 
\int_0^T
\left( \int_{\mathbb{R}^d}||u_{s-}^{(\varepsilon)}|^{p-2}
u_{s-}^{(\varepsilon)}
-|u_{s-}|^{p-2}u_{s-}||h_s|_{\cL_2}\,dx 
\right)^2\,ds                                                       \nonumber\\
& 
\leq 
\int_0^T||u_{s-}^{(\varepsilon)}|^{p-2}u_{s-}^{(\varepsilon)}
-|u_{s-}|^{p-2}u_{s-}|^2_{L_{p/(p-1)}}|h_s|^2_{L_p(\cL_2)}\,ds.                                              \label{A2}
\end{align}
Using \eqref{u sup estimate}, by Lebesgue's theorem on dominated 
convergence we can see that 
$u^{(\varepsilon)}_{s-}=(u_{s-})^{(\varepsilon)}$ for every $s\in(0,T]$ 
and $\omega\in\Omega$.  Since 
$u_{s-}\in L_p(\bR^d)$, we have 
$u^{(\varepsilon)}_{s-}\to u_{s-}$ in $L_p(\bR^d)$ as $\varepsilon\to0$. 
Hence for fixed $\omega$ and $s\in(0,T]$ there is 
a sequence $\varepsilon_k\to0$ 
such that $u^{(\varepsilon_k)}_{s-}(x)\to u_{s-}(x)$ 
for $dx$-almost every $x$, as $k\to\infty$. 
Applying Lemma \ref{lemma tool} to the sequence 
$(u^{(\varepsilon_k)}_{s-})_{k=1}^{\infty}$ in 
$V=L_p$ we get a subsequence, for simplicity 
denoted also by $(u^{(\varepsilon_k)}_{s-})_{k=1}^{\infty}$, and a function 
$v\in L_p$ such that $|u^{(\varepsilon_k)}_{s-}(x)|\leq v(x)$ for all $x\in\bR^d$ and 
all $k$. Thus 
$$
||u_{s-}^{(\varepsilon_k)}|^{p-2}u_{s-}^{(\varepsilon_k)}
-|u_{s-}|^{p-2}u_{s-}|\leq v^{p-1}+|u(s-)|^{p-1}\in L_{p/p-1},  
$$
and by Lebesgue's theorem on dominated convergence 
$$
\lim_{k\to\infty}||u_{s-}^{(\varepsilon_k)}|^{p-2}u_{s-}^{(\varepsilon_k)}
-|u_{s-}|^{p-2}u_{s-}|_{L_{p/p-1}}=0.  
$$
Since we have this for a subsequence of any sequence $\varepsilon_k\to0$, 
we have 
$$
\lim_{\varepsilon\to0}||u_{s-}^{(\varepsilon)}|^{p-2}u_{s-}^{(\varepsilon)}
-|u_{s-}|^{p-2}u_{s-}|_{L_{p/p-1}}=0\quad \text{for all $s\in(0,T]$ 
and $\omega\in\Omega$}.  
$$
By H\"older's inequality we have 
$$
||u_{s-}^{(\varepsilon)}|^{p-2}u_{s-}^{(\varepsilon)}
-|u_{s-}|^{p-2}u_{s-}|_{L_{p/(p-1)}}\leq 2|u_{s-}|_{L_p}^{p-1}\,\text{(a.s.)}. 
$$
Hence 
$$
||u_{s-}^{(\varepsilon)}|^{p-2}u_{s-}^{(\varepsilon)}
-|u_{s-}|^{p-2}u_{s-}|^2_{L_{p/(p-1)}}|h_s|^2_{L_p(\cL_2)}
\leq 2|h_s|^2_{L_p(\cL_2)}\sup_{s\leq T}|u_{s}|_{L_p}^{2p-2}, 
$$
and we can use again Lebesgue's theorem on 
dominated convergence to get that 
$\lim_{\varepsilon\to0}A^2_{\varepsilon}=0$ almost surely. 
Consequently, $\lim_{\varepsilon\to0}A_{\varepsilon}=0$ (a.s.), 
which implies
\begin{equation}                                                                       \label{piconv}
\sup_{t\leq T}\left|\int_0^t\int_Z\int_{\bR^d} 
\left(|u_{s-}^{(\varepsilon)}|^{p-2}u_{s-}^{(\varepsilon)i}h_s^{(\varepsilon)i}
-|u_{s-}|^{p-2}u_{s-}^{i}h_s^{i}\right)\,dx\,\tilde{\pi}(dz,ds)
\right|\to0
\end{equation}
in probability as $\varepsilon\to0$. 
In order to prove $B_\varepsilon\to0$, we are going to show
\begin{equation}                                                                                    \label{J conv 1}
\lim_{\varepsilon\to0}E\int_0^T\int_Z\int_{\bR^d}
\left|J^{h^{(\varepsilon)}_s}|u_{s-}^{(\varepsilon)}|^p-J^{h_s}|u_{s-}|^p\right|
\,dx\,\mu(dz)\,ds=0. 
\end{equation}
Note that by \eqref{u epsilon conv}, \eqref{h epsilon conv} and 
\eqref{f g epsilon conv}, for any sequence $\varepsilon_k\to0$ there is a subsequence, 
denoted also by $\varepsilon_k$, such that  
\begin{equation}                                                                               \label{convptxm}
J^{h^{(\varepsilon_k)}_s}|u_{s-}^{(\varepsilon_k)}|^p\to J^{h}|u_{s-}|^p
\quad
\text{in $P\otimes dt\otimes dx\otimes \mu(dz)$ as $k\to\infty$.}
\end{equation}
Thus to get \eqref{J conv 1} by virtue of Lebesgue's theorem 
on dominated convergence 
we need only show the existence of a function in 
$L_1(\Omega\times[0,T],L_1(\cL_1))$, which dominates 
the integrand in \eqref{J conv 1} for $\varepsilon=\varepsilon_k$ for all 
$k\geq1$. By \eqref{Jestimate} 
$$
J^{h_s^{(\varepsilon)}}|u_{s-}^{(\varepsilon)}|^p
\leq N\big( |u_{s-}^{(\varepsilon)}|^{p-2}
|h_s^{(\varepsilon)}|^2+|h_s^{(\varepsilon)}|^p \big)
$$
with a constant $N=N(p)$. 
Due to  \eqref{h epsilon conv} and \eqref{u epsilon conv}, 
by Lemma \ref{lemma tool}, there exist 
a sequence $\varepsilon_k\rightarrow 0$ 
and functions $v\in\bL_p$ and $H\in\bL_{p,2}$, 
such that together with \eqref{convptxm} 
$$
|u_{s-}^{(\varepsilon_k)}|\leq |v_{s}|\quad \text{and}
\quad |h_s^{(\varepsilon_k)}|\leq |H_s|
\quad \text{for all $(\omega,s,z,x)$ and $k\geq1$}  
$$
hold. Thus 
$$
|J^{h^{(\varepsilon_k)}}|u_{s-}^{(\varepsilon_k)}|^p|
\leq N(|v_s|^{p-2}|H_s|^2+|H_s|^p)\quad 
\text{for $(\omega,s,z,x)$ and $k\geq1$}. 
$$
By H\"older's and Young's inequalities, 
$$
E\int_0^T
\int_{\bR^d}\int_Z
\left(|v_s|^{p-2}|H_s|^2+|H_s|^p
\right)\,\mu(dz)\,dx\,ds
$$
$$
\leq \tfrac{p-2}{p}|v|^p_{\bL_p}
+\tfrac{2}{p}E\int_0^T|H_s|^p_{L_p(\cL_2)}\,ds
+E\int_0^T|H_s|^p_{L_p(\cL_2)}\,ds<\infty,
$$
which shows that 
$|v_s|^{p-2}|H_s|^2+|H_s|^p\in L_1(\Omega\times[0,T],L_1(\cL_1))$ 
and finishes the proof of 
\eqref{J conv 1}. Consequently, 
$$                                    
E\sup_{t\in[0,T]}
\left|\int_0^t\int_Z J^{h^{(\varepsilon)}_s}|u_{s-}^{(\varepsilon)}|^p \,\pi(dz,ds)
-\int_0^t\int_Z J^{h_s}|u_{s-}|^p \,\pi(dz,ds)\right|
$$
$$                                    
\leq E\int_0^T\int_Z \left|J^{h^{(\varepsilon)}_s}|u_{s-}^{(\varepsilon)}|^p
-J^{h_s}|u_{s-}|^p\right| \,\pi(dz,ds)
$$
\begin{equation}                                                   \label{conv3}
=E\int_0^T\int_Z
\left|J^{h^{(\varepsilon)}}|u_{s-}^{(\varepsilon)}|^p-J^{h_s}|u_{s-}|^p\right|
\,\mu(dz)\,ds\to 0\quad \text{as $\varepsilon\to0$}. 
\end{equation} 
Now we are going to show that $\lim_{\varepsilon\to0}C_{\varepsilon}=0$ 
almost surely, which implies 
\begin{equation}                                                             \label{Cconv}
\sup_{t\in[0,T]}\left|\int_0^t\int_{\mathbb{R}^d}
(|u^{(\varepsilon)}_s|^{p-2}u^{(\varepsilon)i}_sg^{(\varepsilon)ir}_s
-|u_s|^{p-2}u_s^ig^{ir}_s)\,dx\,dw^r_s\right|\to0
\end{equation} 
in probability as $\varepsilon\to0$. By Minkowski's inequality 
\begin{equation}                                                              \label{Ce}
C_{\varepsilon}\leq \int_0^T
\left(\int_{\mathbb{R}^d}\left|
|u^{(\varepsilon)}_t|^{p-2}u^{(\varepsilon)i}_tg^{(\varepsilon)i\cdot}_t
-|u_t|^{p-2}u^i_tg^{i\cdot}_t\right|_{l_2}\,dx\right)^2\,dt. 
\end{equation}
Since for fixed $t\in[0,T]$ and $\omega\in\Omega$ we have 
$|u_t^{(\varepsilon)}-u_t|_{L_p}\to0$ 
and $|g_t^{(\varepsilon)}-g_t|_{L_p}\to0$ 
as $\varepsilon\to0$,  for any sequence $\varepsilon_k\to0$ 
there exists a subsequence, denoted also by $\varepsilon_k$, 
such that almost surely
$$
\lim_{\varepsilon_k\to0}u_t^{(\varepsilon_k)}=u_t \quad \text{and}\quad 
\lim_{\varepsilon_k\to0}g_t^{(\varepsilon_k)}=g_t
\quad 
dx\text{-almost everywhere}, 
$$
and, by virtue of Lemma \ref{lemma tool} 
we have functions $v\in L_p$ and $G\in L_p$ 
such that 
$$
|u_t^{(\varepsilon_k)}|\leq v\quad\text{and}
\quad |g_t^{(\varepsilon_k)}|_{\ell_2}\leq G 
\quad\text{for all $x\in\bR^d$ for all $k\geq1$} 
$$
for the fixed $t\in[0,T]$ and $\omega\in\Omega$. Thus
$$
\left||u^{(\varepsilon_k)}_t|^{p-2}|u^{(\varepsilon_k)i}_tg^{(\varepsilon_k)i\cdot}_t
-
|u_t|^{p-2}u_t^ig_t^{i\cdot}\right|_{l_2}\leq v^{p-1}G+|u_t|^{p-1}|g_t|_{\ell_2}
$$
$$
\leq \tfrac{p-1}{p}v^{p}+\tfrac{1}{p}G^p+
\tfrac{p-1}{p}|u_t|^{p}+\tfrac{1}{p}|g_t|^p_{\ell_2}\in L_1(\bR^d,\bR)
\quad\text{for all $k\geq1$}, 
$$
and by Lebesgue's theorem on dominated convergence 
$$ 
\lim_{k\to\infty}\int_{\mathbb{R}^d}\left|
|u^{(\varepsilon_k)}_t|^{p-2}u^{(\varepsilon_k)i}_tg^{(\varepsilon_k)i\cdot}_t
-|u_t|^{p-2}u_t^ig_t^{i\cdot}\right|_{l_2}\,dx=0.  
$$
Consequently, 
$$ 
\lim_{\varepsilon\to0}\int_{\mathbb{R}^d}\left|
|u^{(\varepsilon)}_t|^{p-2}u^{(\varepsilon)i}_tg^{(\varepsilon)i\cdot}_t
-|u_t|^{p-2}u_t^ig_t^{i\cdot}\right|_{l_2}\,dx=0.  
$$
Notice that by H\"older's inequality 
$$
\left||u^{(\varepsilon)}|^{p-2}u^{(\varepsilon)i}g^{(\varepsilon)i\cdot}
-|u|^{p-2}u^ig^{i\cdot}\right|^2_{L_1}
\leq 4|g|^2_{L_p}\sup_{t\in[0,T]}|u_t|_{L_p}^{2p-2}\in L_1([0,T],\bR),  
$$
where $L_1=L_1(\bR^d,\bR)$ and $L_p=L_p(\bR^d,\bR)$. 
Thus letting $\varepsilon \to 0$ in \eqref{Ce} by Lebesgue's theorem 
on dominated convergence we get 
$\lim_{\varepsilon\to0}C_{\varepsilon}=0$. Finally we show
\begin{equation*}                             
E\int_0^T\int_{\bR^d}
\left||u_s^{(\varepsilon)}|^{p-2}u_s^{(\varepsilon)i}f_s^{(\varepsilon)i}-
|u_s|^{p-2}u_s^if_s^i
\right|\,dx\,ds\to0,
\end{equation*}
\begin{equation*}                           
E\int_0^T\int_{\bR^d}
\left||u_s^{(\varepsilon)}|^{p-4}|u_s^{(\varepsilon)i}g^{(\varepsilon)i\cdot}_s|_{l_2}^2
-|u_s|^{p-4}|u^i_sg^{i\cdot}_s|_{l_2}^2\right|
\,dx\,ds\to0, 
\end{equation*}
\begin{equation}                                                              \label{convergence}
E\int_0^T\int_{\bR^d}
\left|
|u_s^{(\varepsilon)}|^{p-2}|g_s^{(\varepsilon)}|_{\ell_2}^2 
-|u_s|^{p-2}|g_s|_{\ell_2}^2
\right|
\,dx\,ds\to0.
\end{equation}
as $\varepsilon\to0$. 
Since $|u^{(\varepsilon)}-u|_{\bL_p}\to0$, 
$|f^{(\varepsilon)}-f|_{\bL_p}\to0$ 
and $|g^{(\varepsilon)}-g|_{\bL_p}\to0$ 
as $\varepsilon\to0$, for any 
sequence $\varepsilon_k\to0$ there exist a subsequence, 
denoted also by $\varepsilon_k$ such that 
\begin{equation*}                                              
\lim_{k\to\infty}(|u^{(\varepsilon_k)}-u|+|f^{(\varepsilon_k)}-f|+
|g^{(\varepsilon_k)}-g|_{\ell_2})=0 \quad P\otimes dt\otimes dx\,(a.e.), 
\end{equation*}
and by virtue of Lemma \ref{lemma tool} there are functions ${\bf v}\in\bL_p$, 
${\bf f}\in\bL_p$ 
and ${\bf g}\in\bL_p$ such that 

$$
|u^{(\varepsilon_k)}|\leq {\bf v},\quad 
|f^{(\varepsilon_k)}|\leq {\bf f},\quad 
|g^{(\varepsilon_k)}|_{\ell_2}\leq {\bf g}
\quad 
\text{for all $(\omega,t,x)\in\Omega\times[0,T]\times\bR^d$ and $k\geq1$}.
$$
Thus by H\"older's inequality 
$$
||u^{(\varepsilon_k)}|^{p-2}u^{(\varepsilon_k)i}f^{(\varepsilon_k)i}-
|u|^{p-2}u^if^i|\leq {\bf v}^{p-1}{\bf f}+|u|^{p-1}|f|\in\bL_1,
$$
$$
||u^{(\varepsilon_k)}|^{p-4}|u^{(\varepsilon_k)i}g^{(\varepsilon_k)i\cdot}|_{l_2}^2
-|u|^{p-4}|u^ig^{i\cdot}|_{l_2}^2|
\leq {\bf v}^{p-2}{\bf g}^2+|u|^{p-2}|g|_{\ell_2}^2\in\bL_1, 
$$
$$
||u^{(\varepsilon_k)}|^{p-2}|g^{(\varepsilon_k)}|_{\ell_2}^2 
-|u|^{p-2}|g|_{\ell_2}^2|
\leq {\bf v}^{p-2}{\bf g}^2+|u|^{p-2}|g|_{\ell_2}^2\in\bL_1, 
$$
and by Lebesgue's theorem on dominated convergence we get 
\eqref{convergence} for $\varepsilon_k\to0$, and hence  
for $\varepsilon\to0$ as well. Using this together with 
\eqref{piconv}, \eqref{conv3} and \eqref{Cconv}, we obtain 
\eqref{Ito1} by letting $\varepsilon\to0$ in \eqref{Itoe}. 
\end{proof}

\begin{proof}[Proof of Theorem \ref{theorem2}]
By taking $\varphi^{(\varepsilon)}$ in place of $\varphi$ in equation \eqref{eq2}, 
we get for each $\varphi\in C_0^\infty$ 
\begin{equation}                                                                    \label{differential}
(u_t^{(\varepsilon)},\varphi)=(\psi^{(\varepsilon)},\varphi)
+\int_0^t(f_s^{(\varepsilon)},\varphi)\,ds+\int_0^t(g_s^{(\varepsilon)r},\varphi)\,dw_s^r
+\int_0^t\int_Z(h_s^{(\varepsilon)},\varphi)\,\tilde{\pi}(dz,ds)
\end{equation}
$P\otimes dt$ almost every $(\omega,t)\in\Omega\times[0,T]$, where
$$
f_s^{(\varepsilon)}:=f_s^{i(\varepsilon)}+f_s^{0(\varepsilon)}, 
$$ 
where $i$ runs through $\{1,2,...,d\}$. 
Hence by Theorem \ref{Ito1} we have an $L_p$-valued adapted cadlag 
process $\bar u^{\varepsilon}$ such that  
for each $\varphi\in C_0^{\infty}$ almost surely 
\eqref{differential} holds with $\bar u^{\varepsilon}$ in place of $u^{(\varepsilon)}$ for all 
$t\in[0,T]$. In particular, for each $\varphi\in C_0^{\infty}$ we have 
$(u^{(\varepsilon)},\varphi)=(\bar u^{\varepsilon},\varphi)$ for $P\otimes dt$-almost 
every $(\omega,t)\in\Omega\times[0,T]$. 
Thus $u^{(\varepsilon)}=\bar u^{\varepsilon}$, as $L_p$-valued functions,  
for $P\otimes dt$-almost every $(\omega,t)\in\Omega\times[0,T]$,  
and almost surely 
$$
|\bar u_t^{\varepsilon}|^p_{L_p}= 
|\psi^{(\varepsilon)}|_{L_p}^p
+p\int_0^t\int_{\mathbb{R}^d}|\bar u_s^{\varepsilon}|^{p-2}
\bar u_s^{\varepsilon}g_s^{(\varepsilon)r}\,dx\,dw^r_s                           
$$
$$
+p\int_0^t\int_{\mathbb{R}^d}
\left(|\bar u_s^{\varepsilon}|^{p-2}\bar u_s^{\varepsilon}
f^{0(\varepsilon)}_s+
|u_s^{(\varepsilon)}|^{p-2}u_s^{(\varepsilon)}D_if_s^{i(\varepsilon)}                
+\tfrac{1}{2}(p-1)
|\bar u_s^{\varepsilon}|^{p-2}|g_s^{(\varepsilon)}|^2_{l_2} \right)\,dx\,ds              
$$
$$
+p\int_0^t\int_Z\int_{\mathbb{R}^d}
|\bar u_{s-}^{\varepsilon}|^{p-2}u_{s-}^{(\varepsilon)}
h_s^{(\varepsilon)}\,dx\,\tilde{\pi}(dz,ds)                                           
+\int_0^t\int_Z\int_{\mathbb{R}^d}
J^{h^{(\varepsilon)}}|\bar u_{s-}^{\varepsilon}|^p\,dx\,\pi(dz,ds)      
$$
for all $t\in[0,T]$. Hence, using that by integration by parts  
$$
\int_{\mathbb{R}^d}
|u_{s}^{(\varepsilon)}|^{p-2}u_{s}^{(\varepsilon)}
D_if_s^{i(\varepsilon)}\,dx
=-\int_{\mathbb{R}^d}(p-1)
|\bar u_s^{\varepsilon}|^{p-2}f^{(\varepsilon)i}_sD_iu_s^{(\varepsilon)}\,dx,    
$$
for $P\otimes dt$-almost every $(\omega,t)\in\Omega\times[0,T]$ we get 
$$
|\bar u_t^{\varepsilon}|^p_{L_p}= 
|\psi^{(\varepsilon)}|_{L_p}^p
+p\int_0^t\int_{\mathbb{R}^d}|\bar u_s^{\varepsilon}|^{p-2}
\bar u_s^{\varepsilon}g_s^{(\varepsilon)r}\,dx\,dw^r_s                           
$$
$$
+p\int_0^t\int_{\mathbb{R}^d}
\left(|\bar u_s^{\varepsilon}|^{p-2}\bar u_s^{\varepsilon}
f^{0(\varepsilon)}_s-(p-1)
|\bar u_s^{\varepsilon}|^{p-2}
f_s^{i(\varepsilon)}D_i u_s^{(\varepsilon)}                  
+\tfrac{1}{2}(p-1)
|\bar u_s^{\varepsilon}|^{p-2}|g_s^{(\varepsilon)}|^2_{l_2} \right)\,dx\,ds              
$$
\begin{equation}                                                        \label{Lp Ito formula smooth}
+p\int_0^t\int_Z\int_{\mathbb{R}^d}
|\bar u_{s-}^{\varepsilon}|^{p-2}\bar u_{s-}^{\varepsilon}
h_s^{(\varepsilon)}\,dx\,\tilde{\pi}(dz,ds)                                           
+\int_0^t\int_Z\int_{\mathbb{R}^d}
J^{h^{(\varepsilon)}}|\bar u_{s-}^{\varepsilon}|^p\,dx\,\pi(dz,ds).       
\end{equation}
almost surely for all $t\in[0,T]$. By Davis', Minkowski 
and H\"older inequalities we have
\begin{align*}
&E\sup_{t\leq T}\left| \int_0^t\int_Z \int_{\mathbb{R}^d} 
 p|\bar u_{s-}^{\varepsilon}|^{p-2}\bar u_{s-}^{\varepsilon}
h_s^{(\varepsilon)}\,dx \,\tilde{\pi}(dz,ds)\right|                                          \nonumber\\
& 
\leq 3pE\left(\int_0^T\int_Z\left(\int_{\mathbb{R}^d}
|\bar u_{s-}^{\varepsilon}|^{p-2}
\bar u_{s-}^{\varepsilon}h_s^{(\varepsilon)}\,dx\right)^2
\,\mu(dz)\,ds \right)^{1/2}                                                                                     \nonumber\\
&
\leq 3pE\left(\int_0^T\left( \int_{\mathbb{R}^d}|\bar u_s^{\varepsilon}|^{p-1}      
|h_s^{(\varepsilon)}|_{\cL_2}\,dx \right)^2\,ds\right)^{1/2}                                                 
\end{align*}
\begin{equation}                                                                                              \label{sup1}
\leq 3pE\left(\int_0^T |\bar u^{\varepsilon}_s|_{L_p}^{2p-2}
|h_s^{(\varepsilon)}|^2_{L_p(\cL_2)}\, ds \right)^{1/2}                                          
\leq \frac{1}{12}\,E\sup_{t\leq T}|\bar u^{\varepsilon}_{t}|^p_{L_p}                           
+NT^{(p-2)/2}|h^{(\varepsilon)}|^p_{\bL_{p,2}},
\end{equation}
with a constant  $N=N(p,d)$. 
Similarly,
\begin{equation}                                                                               \label{sup2}
E\sup_{t\leq T}\left|\int_0^t\int_{\mathbb{R}^d}p|\bar u_s^{\varepsilon}|^{p-2}
\bar u_s^{\varepsilon}g_s^{(\varepsilon)r}\,dx\,dw^r_s\right| 
\leq
\frac{1}{12}\,E\sup_{t\leq T}|\bar u^{\varepsilon}_{t}|^p_{L_p}
+NT^{(p-2)/2}|g^{(\varepsilon)}|^p_{\bL_{p}}. 
\end{equation}
By \eqref{Jestimate} and H\"older inequality we have
$$
E\int_0^T\int_Z\int_{\mathbb{R}^d}  J^{h^{(\varepsilon)}}
| \bar u^{\varepsilon}_{s-}|^p\,dx\,\pi(dz,ds)                                                       \nonumber 
$$
$$
\leq N\,E\int_0^T\int_{\mathbb{R}^d}\int_Z\big( |\bar u_{s-}^{\varepsilon}|^{p-2}
|h_s^{(\varepsilon)}|^2+|h_s^{(\varepsilon)}|^p \big)\,\mu(dz)\,dx\,ds                
$$
\begin{equation}                                                            \label{sup3}
\leq \frac{1}{12}\,E\sup_{t\leq T}|\bar u^{\varepsilon}_{t}|^p_{L_p}                             
+N^{'}T^{(p-2)/2}
E\int_0^T|h_t^{(\varepsilon)}|^p_{L_p(\cL_2)}\,dt
+N E\int_0^T|h_t^{(\varepsilon)}|^p_{L_p(\cL_p)}\,dt, 
\end{equation}
with constants $N$ and $N^{'}$ depending only on $p$ and $d$. 
By H\"older's and Young's inequalities 
$$
pE\int_0^T\int_{\mathbb{R}^d}
|\bar u_s^{\varepsilon}|^{p-2}|\bar u_s^{\varepsilon}
f^{0(\varepsilon)}_s|\,dx\,ds
\leq \frac{1}{12}\,E\sup_{t\leq T}|\bar u^{\varepsilon}|^p_{L_p}
+NT^{p-1}|f^{0(\varepsilon)}_s|^p_{\bL_p}
$$
$$
p(p-1)E
\int_0^T\int_{\bR^d}
|\bar u_s^{\varepsilon}|^{p-2}
f_s^{i(\varepsilon)}D_i u_s^{(\varepsilon)}\,dx\,ds
$$
$$
\leq  
\frac{1}{12}\,E\sup_{t\leq T}|\bar u^{\varepsilon}|^p_{L_p}
+NT^{(p-2)/2}
\Big(
\sum_{i=1}^d|f^{i(\varepsilon)}|_{\bL_p}^p
+|Du^{(\varepsilon)}|^p_{\bL_p}
\Big)
$$
\begin{equation*}                                                       \label{sup4}            
\tfrac{1}{2}p(p-1)E
\int_0^T\int_{\bR^d}
|\bar u_s^{\varepsilon}|^{p-2}|g_s^{(\varepsilon)}|^2_{l_2}\,dx\,ds
\leq 
\frac{1}{12}\,E\sup_{t\leq T}|\bar u^{\varepsilon}_{t}|^p_{L_p}
+NT^{(p-2)/2}|g^{(\varepsilon)}|_{\bL_p}^p
 \end{equation*}
Using these inequalities together with \eqref{sup1}, \eqref{sup2} 
and \eqref{sup3}, 
from \eqref{Lp Ito formula smooth} 
we obtain  
$$        
E\sup_{t\leq T}|\bar u_t^{\varepsilon}|^p_{L_p} 
\leq 2\,E|\psi^{(\varepsilon)}|_{L_p}^p
+NE\int_0^T|h_t^{(\varepsilon)}|^p_{L_p(\cL_p)}\,dt
+NT^{p-1}
|f^{0(\varepsilon)}|^p_{\bL_p}                                                              \nonumber\\
$$
\begin{equation}                                                   \label{Ito Lp estimate smooth}
+ NT^{(p-2)/2}\big(|g^{(\varepsilon)}|^p_{\bL_p}
+E\int_0^T|h_t^{(\varepsilon)}|^p_{L_p(\cL_2)}
+\sum_{i=1}^d|f^{i(\varepsilon)}|_{\bL_p}^p+|Du^{(\varepsilon)}|^p_{\bL_p}
\big)                                                  
\end{equation}
with a constant $N=N(p,d)$. Hence  
$$
E\sup_{t\leq T}|\bar u_t^{\varepsilon}-\bar u_t^{\varepsilon'}|^p_{L_p}
\rightarrow 0
\quad
\text{as $\varepsilon,\,\varepsilon'\to0$}.  
$$ 
Consequently, there is an $L_p$-valued adapted cadlag process 
$\bar u=(\bar u_t)_{t\in[0,T]}$ 
such that 
$$
\lim_{\varepsilon\to0}E\sup_{t\leq T}|\bar u_t^{\varepsilon}-\bar u|^p_{L_p}=0. 
$$
Thus for each $\varphi\in C_0^{\infty}(\bR^d)$ we can take $\varepsilon\to0$ in 
\begin{align*}
(\bar u_t^{\varepsilon},\varphi)
& =(\psi^{(\varepsilon)},\varphi)+\int_0^t(f_s^{(\varepsilon)},\varphi)\,ds
+\int_0^t(g_s^{(\varepsilon)r},\varphi)\,dw_s^r
+\int_0^t\int_Z(h_s^{(\varepsilon)},\varphi)\,\tilde{\pi}(dz,ds)\\
&
= (\psi^{(\varepsilon)},\varphi)+\int_0^t(f_s^{0(\varepsilon)},\varphi)\,ds
-\int_0^t(f_s^{i(\varepsilon)},D_i\varphi)\,ds
+\int_0^t(g_s^{(\varepsilon)r},\varphi)\,dw_s^r\\
&
\quad +\int_0^t\int_Z(h_s^{(\varepsilon)},\varphi)\,\tilde{\pi}(dz,ds)
\end{align*}
and it is easy to see we get 
$$
(\bar u_t,\varphi)=(\psi,\varphi)+\int_0^t(f_s^\alpha,D^*_\alpha\varphi)\,ds
+\int_0^t(g_s^r,\varphi)\,dw_s^r+\int_0^t\int_Z(h_s,\varphi)\,\tilde{\pi}(dz,ds) 
$$
almost surely for all $t\in[0,T]$. Hence $\bar u= u$ for $P\otimes dt$-almost every 
$(\omega,t)\in\Omega\times[0,T]$. 
Letting $\varepsilon\to0$ 
in \eqref{Ito Lp estimate smooth}, we get estimate \eqref{Ito Lp estimate}. 
Finally letting $\varepsilon\to0$ in \eqref{Itoe}, by analogous arguments 
as in the proof  of Theorem \ref{Ito1}, we obtain \eqref{Ito2}. 
\end{proof}

\smallskip
\noindent 
 {\bf Acknowledgement.} The main results of the paper 
were presented at the conference on 
``New Directions in Stochastic Analysis: Rough Paths, SPDEs and Related Topics", 
on the occasion of Professor Terry Lyons' 65th Birthday, in Zuse Institute Berlin,  
8-22 March, 2019. The authors are grateful to the organisers for the invitation.

\end{document}